\documentclass[11pt]{article}

\usepackage{tikz}
\usepackage{subfigure}
\usepackage[english]{babel}
\usepackage{graphicx}
\usepackage{amsthm}
\usepackage{appendix}
\usepackage{booktabs}
\usepackage{enumitem}
\usepackage{indentfirst}
\usepackage{enumitem}
	\usepackage{titlesec}
\usepackage[center]{caption2}
\usepackage{amsfonts,amssymb,amsmath,latexsym,amsthm}
\usepackage{multirow}

\usepackage{geometry}
\titleformat{\paragraph}
  {\normalfont\normalsize\bfseries}{\theparagraph}{1em}{}
\titlespacing*{\paragraph}{0pt}{3.25ex plus 1ex minus .2ex}{1.5ex plus .2ex}

\def\sat{\mbox{sat}}

\newtheorem{lem}{Lemma}[section]
\newtheorem{myDef}{Definition}
\newtheorem{thm}{Theorem}[section]

\newtheorem{prop}[lem]{Proposition}
\newtheorem{conj}{Conjecture}[section]

\newtheorem{claim}{Claim}

\newtheorem{remark}{Remark}

\makeatletter
\renewenvironment{proof}[1][\proofname]{\par
	\pushQED{\qed}%
	\normalfont \topsep6\p@\@plus6\p@\relax
	\trivlist
	\item[\hskip\labelsep
	\bfseries #1\@addpunct{.}]\ignorespaces
}{%
	\popQED\endtrivlist\@endpefalse
}
\makeatother

\newenvironment{mycase}[1]
{\par\vspace{0.5\baselineskip}\noindent\textbf{Case #1.}\hspace{0.5em}\ignorespaces}
{\par}

\begin{document}
	
	\title{Minimizing the number of edges in $\mathcal{C}_{[4,6]}$-saturated graphs}
	\author{Qi Liu,  Dijian Wang\thanks {Corresponding author: dijianwang0121@163.com},  Shicai Gong \\
	\small  School of Science, Zhejiang University of Science and Technology, \\
	\small Hangzhou, Zhejiang, 310023, P. R. China\\}\date{}
	
	\baselineskip=0.30in
	
	\date{}
	
	\maketitle
	
	\begin{abstract}
		Let $\mathcal{C}_{[4,r]}$  be the family of cycles $\{C_4, \dots, C_r\}$. 	
		A graph $G$ is said to be $\mathcal{C}_{[4,r]}$-saturated if $G$ does not contain a copy of cycle $C_i$ for $4\le i\le r$, but the addition of any edge $e\notin E(G)$ creates at least one copy of $C_i$ for $4\le i\le r.$ The saturation number $\sat(n, \mathcal{C}_{[4,r]})$ is the minimum number of edges in an $n$-vertex  $\mathcal{C}_{[4,r]}$-saturated graph.
		In 2025,  Ma  determined that $\sat(n, \mathcal{C}_{[4,5]})=\lceil \frac{5n}{4} - \frac{3}{2} \rceil$, and conjectured that for any $r \ge 5$, $\sat(n, \mathcal{C}_{[4,r]}) = \lceil\frac{5n}{4} - \frac{3}{2} \rceil$ holds for large $n$. In this paper we  prove that $\sat(n, \mathcal{C}_{[4,r]}) \le \lceil\frac{5n}{4} - \frac{r+1}{4}\rceil$ for  $n \ge r+1$, which disproves 
		Ma's conjecture for $r\ge 6.$ For $r=6,$ we  determine that $\sat(n, \mathcal{C}_{[4,6]})=\lceil\frac{5n}{4}-\frac{7}{4}\rceil.$
	\end{abstract}
	
	\section{Introduction}
		Let $G = (V(G),E(G))$ be an $n$-vertex graph with order $|V(G)|$ and size $|E(G)|$.  For any $v \in V$, we use $N_G(v)$ to denote the set of the neighbors of $v$ in $G$ and let $d_G(v) = |N_G(v)|$ be the degree of $v$ in $G$. The minimum degree of $G$ is
		denoted by $\delta(G)$. For a set $U \subseteq V(G)$, let $G[U]$ be the subgraph induced by $U.$ Write $G \setminus U = G[V(G) \setminus U]$ and $N_G(U)=\bigcup_{v\in U}N_G(v)$. For any two disjoint sets $U_1, U_2 \subseteq V(G)$, let $E(U_1,U_2) = \{u_1u_2 \in E(G): u_1 \in U_1 \text{ and } u_2 \in U_2\}$. Moreover, let $e(G)=|E(G)|$ and $e(U_1,U_2)=|E(U_1,U_2)|.$
		
		Given a family $\mathcal{F}$ of graphs, a graph $G$ is said to be $\mathcal{F}$-saturated if $G$ does not contain a subgraph isomorphic to any member $F \in \mathcal{F}$ but $G+e$ contains at least one copy of some $F \in \mathcal{F}$ for any edge $e \not\in E(G)$. The minimum number of edges in an $n$-vertex $\mathcal{F}$-saturated graph, which is called the saturation number, denoted by $\sat(n, \mathcal{F})$.
		In 1961, Erd\H{o}s, Hajnal and Moon~\cite{EHM64} first studied on the saturation number of $K_{r+1}$, they
		determined $\sat(n,K_{r+1})=(r-1)(n-r+1)+\binom{r-1}{2}$ and characterized the extremal graphs.
		Since then,	determining $\sat(n,\mathcal{F})$ for given $\mathcal{F}$ and $n$, and characterizing the $n$-vertex $\mathcal{F}$-saturated graph $G$ with $\sat(n,\mathcal{F})$ edges become the challenging problems in extremal graph theory. For history and developments on the theory of saturation number, we refer the reader to the dynamic survey~\cite{CF21} by Currie, Faudree, Faudree, and Schmitt.
		
		Let $P_{x_1x_r}=x_1x_2\dots x_r$ denote the path  of length $r-1$ with edge set $\{x_1x_2,x_2x_3,\dots,x_{r-1}x_{r}\}$. Let $C_r$ denote the cycle of length $r$ for $r\ge 3$. 
		For integers $s$ and $r$ with $s \le r$, let $[s,r] = \{s,s + 1, s + 2, \dots , r\}$.  For any  $I=[s,r]$, let $\mathcal{C}_I$ be the family of all cycles of length $\ell \in I$.
		
		If $s=r,$ the $\mathcal{C}_I$ is a single cycle $C_r.$ There are many known results for $\sat(n, C_r)$, such as:
		
		$\bullet$ $\sat(n,C_3)=n-1$ for $n\ge 3$; (Erd\H{o}s, Hajnal and Moon~\cite{EHM64})
		
		$\bullet$  $\sat(n, C_4)=\lfloor\frac{3n-5}{2}\rfloor$ for $n\ge 5$; (Ollmann \cite{Oll72}, Tuza~\cite{Tuz89}, Fisher, Fraughnaugh and Langley~\cite{FFL97}) 
		
		$\bullet$  $\sat(n, C_5)=\lceil\frac{10}{7}(n-1)\rceil$ for $n\ge 21$; (Chen~\cite{Che09,Che11})
		
		$\bullet$ $\frac{4n}{3}-2\le\sat(n,C_6)\le\frac{4n+1}{3}$ for $n\ge 9$; (Lan, Shi, Wang and Zhang~\cite{Lan21})
		
		$\bullet$  $(1+\frac{1}{r+2})n-1<\sat(n,C_r)<(1+\frac{1}{r-4})n+\binom{r-4}{2}$ for all $r\ge 7$ and $n\ge 2r-5$; (F\"{u}redi and Kim~\cite{FK13}) 
		
		$\bullet$
		$\sat(n,C_n)=\lceil\frac{3n}{2}\rceil$ for $n=17$ or $n\ge 19$. (Clark, Entringer and Shapiro~\cite{Clark83, Clark92}, Lin, Jiang, Zhang and Yang \cite{LJZY97}) 
		
		If $r=+\infty,$ here are some known exact results for cycle families $\mathcal{C}_{[s,+\infty)}$:
		
		$\bullet$  $\sat(n, \mathcal{C}_{[4,+\infty)})=\lceil\frac{5n}{4}-\frac{3}{2}\rceil$ for $n\ge 1$; (Ferrara et al.~\cite{Subdivision12})
		
		$\bullet$  $\sat(n, \mathcal{C}_{[5,+\infty)})=\lceil\frac{10(n-1)}7\rceil$ for $n\ge 5$; (Ferrara et al.~\cite{Subdivision12})
		
		$\bullet$  $\sat(n, \mathcal{C}_{[6,+\infty)})=\lceil\frac{3(n-1)}2\rceil$ for $n\ge 10$. (Ma, Hou, Hei and Gao~\cite{MHHG21})
		
		If $s=4$ and $r=5,$
		Ma \cite{Ma25} (2025) determined the saturation number on $\mathcal{C}_{[4,5]}.$

	\begin{thm}\cite[Theorem 1.2]{Ma25}
		For $n\ge 1,$ $\sat(n, \mathcal{C}_{[4,5]}) = \lceil \frac{5n}{4} - \frac{3}{2} \rceil.$
	\end{thm}	
		Note that $\sat(n, \mathcal{C}_{[4,+\infty)})=\sat(n, \mathcal{C}_{[4,5]}) = \lceil \frac{5n}{4} - \frac{3}{2} \rceil,$ Ma \cite{Ma25} conjectured $\sat(n, \mathcal{C}_{[4,r]})= \lceil \frac{5n}{4} - \frac{3}{2} \rceil$ for any integer $r \ge 5$.

	\begin{conj}\cite[Conjecture 3.1]{Ma25}\label{conjecture1}
		For any integer $r \ge 5$, there exists a number $n(r)$, such that for any integer $n \ge n(r)$, $\sat(n, \mathcal{C}_{[4,r]})=\lceil \frac{5n}{4} - \frac{3}{2} \rceil$.
	\end{conj}

	In this paper, we first  prove that $\sat(n, \mathcal{C}_{[4,r]}) \le \lceil\frac{5n}{4} - \frac{r+1}{4}\rceil$ for $r\ge 6$ and  $n \ge r+1$, which disproves Conjecture~\ref{conjecture1} for $r\ge 6.$	
	\begin{thm}\label{thm_c4r}
		For $r\ge 6$ and  $n \ge r+1$, $sat(n, \mathcal{C}_{[4,r]}) \le \lceil\frac{5n}{4} - \frac{r+1}{4}\rceil$.
	\end{thm}

Moreover, we prove that the bound in Theorem \ref{thm_c4r} is tight for $r=6$.

	\begin{thm}\label{thm_sat_C456}
		For $n \ge 7$, $\sat(n, \mathcal{C}_{[4,6]})=\lceil\frac{5n}{4} - \frac{7}{4}\rceil$.
	\end{thm}

	The rest of the paper is arranged as follows. In Section 2, the upper bound of $\sat(n, \mathcal{C}_{[4,r]})$ will be given. In Section 3, we give the proof of Theorem~\ref{thm_sat_C456}. In Section 4, there are some remarks for $\mathcal{C}_{[4,r]}$-saturated graphs.
	
	\section{The Upper Bound of $\sat(n, \mathcal{C}_{[4,r]})$}
	
	Let $F_k$  denote the graph on $2k+1$ vertices consisting
	of $k$ triangles all sharing exactly one same vertex $v_0.$ Let $F_k^+$ be the graph on $4k + 1$ vertices obtained by adding one pendant edge to each vertex in $F_k$ except $v_0$. See  Figure~\ref{fig_graph_Fk+}.
	Given a cycle $C_{r+1} = v_0v_1v_2 \dots v_{r}v_0$ of length $r+1$. Let $H_{r,k}$ be a graph obtained from $C_{r+1}$ and $F_k^+$ by  identifying vertex $v_0$ of graph $C_{r+1}$  and vertex $v_0$ of graph $F_k^+$.
	We define the graph $G_{r,k}$ on $n = (4k + r - 1) + d$ vertices as follows:
	if $n=r+1,$ then $G_{r,k}=C_{r+1};$ if $n=r+2,$ then $G_{r,k}=C_{r+2}^+;$ if $n\ge r+3$, then
	\[
	G_{r,k}=\left\{
	\begin{aligned}
		&H_{r,k}\setminus \{u_1^\prime, w_1^\prime\}& ~~\text{if $d=0,$}\\
		&H_{r,k} \setminus \{w_1^\prime\} & ~~\text{if $d=1,$}\\
		&H_{r,k} & ~~\text{if $d=2,$}\\
		& H_{r,k} \setminus \{w_1^\prime,u_2^\prime,w_2^\prime\} &~~\text{if $d=3.$}
	\end{aligned}
	\right.
	\]
	See Figure~\ref{fig_graph_Grk}.
	Note that $e(G_{r,k}) =\lceil\frac{5n}{4} - \frac{r+1}{4}\rceil$. 	
		\begin{figure}
		\centering
		\includegraphics[width=4cm,height=1.75cm]{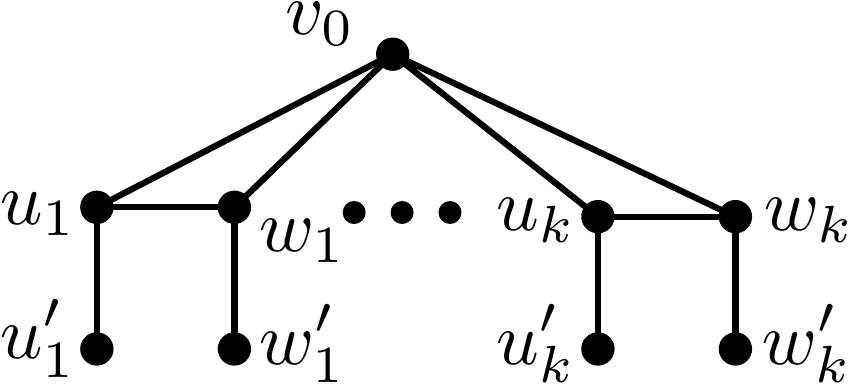}
		\caption{ Graph $F_k^+$}
		\label{fig_graph_Fk+}
	\end{figure}
	
	\begin{figure}
		\centering
		\includegraphics[width=14.5cm,height=2.75cm]{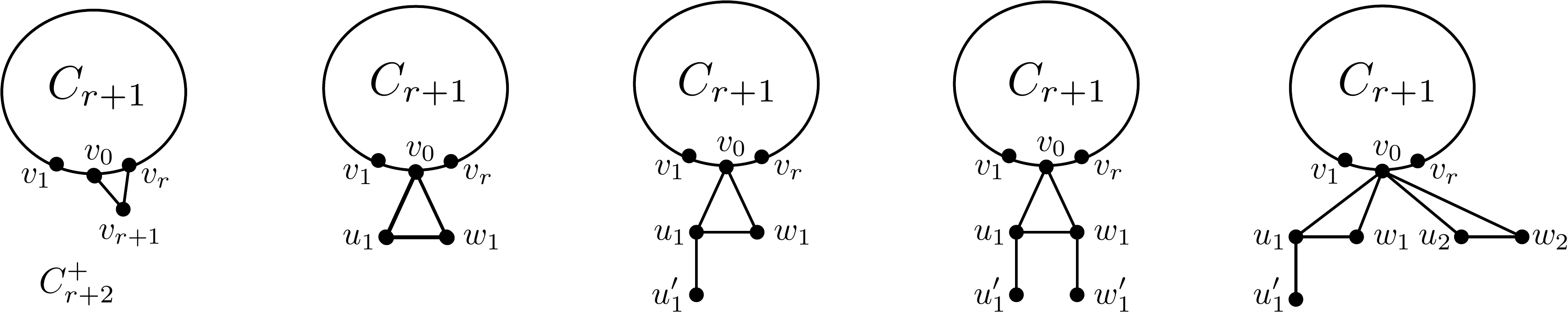}
		\caption{ Graph $G_{r,k}$ for $n = r+2, r+3,r+4,r+5,r+6$}
		\label{fig_graph_Grk}
	\end{figure}

	\begin{proof}[Proof of Theorem~\ref{thm_c4r}]
		We just prove that $G_{r,k}$ is a $\mathcal{C}_{[4,r]}$-saturated graph on $n$ vertices. 
        If $n=r+1$ or $n=r+2$, then $C_{r+1}$ and $C_{r+2}^+$ are clearly $\mathcal{C}_{[4,r]}$-saturated.
		If $n\ge r+3,$ then $V(G_{r,k})=V(C_{r+1})\cup (V(F_k^+)\setminus \{v_0\}).$
		For any vertex $v\in V(C_{r+1})$ (resp. $v\in V(F_k^+)$), there is a path $P_{vv_0}$ of length at most  $\lceil \frac{r}{2} \rceil$ (resp. 2)
		connecting $v$ and $v_0.$
		For any two non-adjacent vertices $v,u$ in $G$, there is a path $P_{vu}=P_{vv_0}\cup P_{v_0u}$ connecting $v$ and $u.$
		Thus, $G+vu$ contains a cycle $C_{\ell}$ of length $\ell=|P_{vv_0}|+|P_{v_0u}|+1.$ 
		Since  $v,u$ are non-adjacent, then $|P_{vv_0}|\le \lceil \frac{r}{2} \rceil-1$ or $|P_{v_0u}|\le \lceil \frac{r}{2} \rceil-1$, which implies that $\ell\le \lceil \frac{r}{2} \rceil+\lceil \frac{r}{2} \rceil-1+1\le r$.
		It follows that $G_{r,k}$ is  $\mathcal{C}_{[4,r]}$-saturated. Hence,  $\sat(n, \mathcal{C}_{[4,r]})  \le e(G_{r,k})=\lceil\frac{5n}{4} - \frac{r+1}{4}\rceil.$ 
	\end{proof} 
	
	\section{Proof of Theorem~\ref{thm_sat_C456}}
		By Theorem \ref{thm_c4r}, we have known that $\sat(n, \mathcal{C}_{[4,6]}) \le \lceil\frac{5n}{4} - \frac{7}{4}\rceil.$
		In this section, we will prove that $\sat(n, \mathcal{C}_{[4,6]}) \ge  \frac{5n}{4} - \frac{7}{4}.$
		In the following of this section, we suppose that $G$ is a $\mathcal{C}_{[4,6]}$-saturated graph on $n$ vertices. 
		The following five facts are  trivial by the saturation of $G.$
		
		\textbf{Fact 1.} $G$ is connected.
		
		\textbf{Fact 2.}  For any vertex $v\in V (G),$  $G[N_G (v)]=tK_2\cup sK_1$.
		
		\textbf{Fact 3.} Any vertex in $G$ has at most one  neighbor of degree $1$.

		Let $\mathcal{Q} = \{ v \in V(G) : v \text{ has a neighbor of degree } 1 \}$. For any vertex $v \in \mathcal{Q}$, denote by $p_v$ its unique neighbor of degree 1.
		
		\textbf{Fact 4.} If any vertex $v\in \mathcal{Q},$ then  $G[N_G (v)\setminus\{p_v\}]=tK_2.$

		\textbf{Fact 5.} Any pair of non-adjacent vertices $u,v$ is connected by a path of length $3 ,4$ or $5.$
		
		For a fixed vertex $v \in V(G)$, let $L^i$ denote the set of vertices at distance $i$ from $v$. By Fact 5, $L^i = \emptyset$ for all $i \ge 6$. Thus the vertex set of $G$ is partitioned as $V(G) = L^0 \cup L^1 \cup L^2 \cup L^3 \cup L^4 \cup L^5,$ where $L^0=\{v\}$. We now define the following three properties for a subset $V' \subseteq V(G)$:
	
		$\bullet$	\textbf{Property 1:} $L^0 \cup L^1 \cup L^2 \cup L^3 \subseteq V^\prime$;
			
		$\bullet$	\textbf{Property 2:} $e(V^\prime) \ge \frac{5}{4}|V^\prime| -\frac{7}{4}$;
			
		$\bullet$	\textbf{Property 3:} if any vertex $v \in V^\prime \cap \mathcal{Q}$, then $p_v \in V^\prime$.
	
		The following lemma plays a central role in the proof of Theorem \ref{thm_sat_C456}; the proof of the lemma is presented in Section~3.1.		
		\begin{lem}\label{lem_existence_of_Vstar}
			Suppose that $v$ is a vertex of minimum degree in $G$ and $L^0=\{v\}$, where $d_G(v)=1$ or $2$. There exists a vertex subset $V^{\star} \subseteq V(G)$ satisfying Properties $1, 2,$ and 
			
			\textbf{$\bullet$ Property $3^+$:} $d(u) \ge 2$ for any vertex $u \in V(G) \setminus V^{\star}$.
		\end{lem}
		By Lemma \ref{lem_existence_of_Vstar}, we have  $V(G)=V^{\star} \cup (L^4 \setminus V^{\star})\cup  (L^5 \setminus V^{\star}).$
		\begin{myDef}
			For any vertex $z \in L^i \setminus V^{\star}$, $i\in \{4,5\},$ define $w(z)$ as the weight of $z:$
			\begin{equation*}\label{weight}
				w(z)=|N_G(z) \cap (L^{i-1} \cup V^{\star})|+\frac{1}{2}|N_G(z) \cap (L^i \setminus V^{\star})|.
			\end{equation*}
		\end{myDef}
		Observe that $w(z) \ge 1$ for any vertex $z \in V (G)  \setminus V^{\star}$, as $|N_G(z) \cap L^{i-1}| \ge 1$.
		
		Now we can calculate $e(G)$ by following equation:
		\begin{equation}\label{eG}
			e(G) = e(V^{\star})+\sum_{z \in  V (G) \setminus V^{\star}} w(z).
		\end{equation}
	
		We are ready to give the  proof of Theorem~\ref{thm_sat_C456}.
	
		\begin{proof}[Proof of Theorem~\ref{thm_sat_C456}] 
		If $\delta(G)\ge 3,$ then $e(G)\ge \frac{3}{2}n>\frac{5}{4}n - \frac{7}{4}$ for $n\ge 7.$ 
		Then we  consider that $\delta(G)=1$ or 2.  Suppose that $v$ is a vertex of minimum degree in $G$ and $L^0=\{v\}.$
		Let  vertex subset $V^\star\subseteq V(G)$ satisfy Properties 1, 2 and $3^+$.
		We now define  the vertex set $S_1$ as follows:
		$
			S_1 = \{z \in V(G) \setminus V^{\star} : w(z) = 1 \},
	$ that is, $|N_G(z) \cap (L^{i-1} \cup V^{\star})|=1$ and $|N_G(z) \cap (L^i \setminus V^{\star})|=0$.
			Then,
			for any $z \in V(G)  \setminus (V^{\star} \cup S_1)$, we have $w(z) \ge 3/2$. For any $z \in L^5 \setminus V^{\star}$, since $d(z) \ge 2$ (by Property $3^+$), it follows that 
			 $|N_G(z) \cap (L^4 \cup V^{\star})| \ge 2$
			or 
			$|N_G(z) \cap (L^5 \setminus V^{\star})| \ge 1,$
			which yields $w(z) \ge 3/2$. Consequently, no vertex of $L^5 \setminus V^{\star}$ lies in $S_1$, and hence $S_1 \subseteq L^4 \setminus V^{\star}$.
					
		Denote that $S_2=N_{L^{5} \setminus V^{\star}}(S_1)=\{z_1,\dots,z_k\}.$ Then 
			\begin{equation} \label{eq3}
				\begin{split}
				\sum_{z \in V (G) \setminus V^{\star}} w(z)= & \sum_{z \in L^{4} \setminus V^{\star}} w(z)+\sum_{z \in L^{5} \setminus V^{\star}} w(z)\\
				=&~ (\sum_{z \in S_1} w(z)+
				\sum_{z \in L^{4} \setminus (V^{\star}\cup  S_1)} w(z))+ (\sum_{z \in S_2} w(z)+
				\sum_{z \in L^{5} \setminus (V^{\star}\cup S_2)} w(z))\\
				\ge & \sum_{z \in S_1\cup S_2} w(z)+
				\sum_{z \in V(G) \setminus (V^{\star}\cup S_1\cup S_2)} \frac{3}{2}.
				\end{split}
			\end{equation}
	Let vertex sets $Z_1=\{z_1\}\cup N_{S_1}(z_1)$  and	
	$Z_i=\{z_i\}\cup \big(N_{S_1}(z_i)\setminus \bigcup_{j=1}^{i-1}N_{S_1}(z_j)\big)$ for $i=2,\dots,k,$ then $S_1\cup S_2=\cup_{i=1}^k Z_i$ and $\sum_{z \in S_1\cup S_2} w(z)=\sum_{i=1}^k\sum_{z \in Z_i}w(z)$. Moreover, we have 

	(i) if $|Z_i|=1,$ then $\sum_{z \in Z_i}w(z)=w(z_i)\ge \frac{3}{2}=\frac{3}{2}|Z_i|>\frac{5}{4}|Z_i|;$

	(ii) if $|Z_i|=2,$ then $\sum_{z \in Z_i}w(z)=w(z_i)+1\ge \frac{5}{2}=\frac{5}{4}|Z_i|;$ 

	(iii) if $|Z_i|=\ell \ge 3,$ then $w(z_i)\ge \ell-1$ and $\sum_{z \in Z_i}w(z)=\ell-1+\ell-1=2\ell-2 \ge \frac{5}{4}\ell=\frac{5}{4}|Z_i|.$ 

	From the above (i), (ii), (iii), and by Lemma~\ref{lem_existence_of_Vstar}, Eqs. (\ref{eG}) and (\ref{eq3}), we conclude that
	\begin{equation*} 
		\begin{split}
		e(G) = & ~e(V^{\star})+\sum_{z \in  V(G) \setminus V^{\star}}w(z)\ge  \frac{5}{4}|V^{\star}| - \frac{7}{4} +\sum_{z \in  V(G) \setminus V^{\star}}w(z)
		~~\text{(by Eq. (\ref{eG}) and Lemma \ref{lem_existence_of_Vstar})}\\
	\ge &~\frac{5}{4}|V^{\star}| - \frac{7}{4}+  
	  \sum_{i=1}^k\sum_{z \in Z_i}w(z)+ \frac{3}{2}(|V(G)|-  |V^{\star}|- |S_1|-|S_2|)~~\text{(by Eq. (\ref{eq3}))}\\
	  \ge &~\frac{5}{4}|V^{\star}| - \frac{7}{4}+  
	  \frac{5}{4}\sum_{i=1}^k|Z_i|+ \frac{5}{4}(|V(G)|-  |V^{\star}|- |S_1|-|S_2|)~~\text{(by (i), (ii), (iii))}\\ 
	  =&~\frac{5}{4}|V^{\star}| - \frac{7}{4} + \frac{5}{4} (n-|V^{\star}|)\\
	  =&~ \frac{5}{4}n - \frac{7}{4}.
	  \end{split}
	\end{equation*}

	Combining with Theorem \ref{thm_c4r}, we complete the proof of  Theorem~\ref{thm_sat_C456}.
	\end{proof}
  	
	\subsection{Proof of Lemma~\ref{lem_existence_of_Vstar}}
	We need to consider  two cases: $\delta(G) = 1$ or 2. In each case, we will show how to construct a set $V^\star$ satisfying Properties 1, 2 and $3^+$.			
	The following result will be  frequently used.
	
  	\begin{prop}\label{prop1}
    	If $x \notin \mathcal{Q}$ and $u$, $v$ are two isolated vertices in $G[N_G(x)]$, then $x,$ $u$ and $v$ lie on some induced $C_7$. 

  	\end{prop}
  	
  	\begin{proof}
  			By Fact 5, there is a path  of length $r\in \{3, 4, 5\}$ connecting $u$ and $v.$
  		Together with two edges $xu$ and $xv$, $G$ contains a cycle of length  $r+2.$ Since $G$ is $\mathcal{C}_{[4,6]}$-saturated, then $r =5.$
  		Hence, $x,$ $u$ and $v$ lie on some induced $C_7$. 
  	\end{proof}
  	
	For each vertex $x \in L^{i-1}$ ($i \in \{1,2,3,4,5\}$), let $L^i(x) = L^i \cap N_G(x)$.
	By Fact~2, every component of the induced subgraph $G[L^i(x)]$ is either a $K_2$ or a $K_1$. Define that
	\[
	L^i(x) = L^i_P(x) \cup L^i_F(x),\qquad
	L^i_P = \bigcup_{x\in L^{i-1}} L^i_P(x),\qquad
	L^i_F = \bigcup_{x\in L^{i-1}} L^i_F(x),
	\]
	where $L^i_P(x)$ (resp.\ $L^i_F(x)$) denotes the vertex set of the $K_2$ (resp.\ $K_1$) components of $G[L^i(x)]$.

	\subsubsection{$\delta(G) = 1$}
		In this case, we select any vertex $\alpha_0$ of degree 1. Suppose that $N_G(\alpha_0) =\{\alpha_1\}$, then $L^0 = \{\alpha_0\}$ and $L^1 = \{\alpha_1\}$. By Fact 4, we have	$L^2_{F}(\alpha_1) = \emptyset$.
	The following result shows that 	 any vertex not lying in $L^5$ has a unique neighbor in the preceding layer.

			\begin{prop}\label{prop3}
			If the vertex $x \in L^i$ and $|N_G(x) \cap L^{i-1}| \ge 2$, then $i = 5$.
		\end{prop}	
		
		\begin{proof}
			It is trivial that $i \notin \{0,1,2\}$.  If $x \in L^3$ and $|N_G(x) \cap L^{2}| \ge 2,$ it is easy to see that there is a $C_4$ in $G$, a contradiction.
			If $x \in L^4$ and $|N_G(x) \cap L^{3}| \ge 2$, assume that 
			$\{y_1, y_2\} \subseteq N_G(x) \cap L^{3}$ and $y_i \in L^3(x_i)$ for $i \in \{1,2\},$
			then $xy_1x_1y_2x$ forms a $C_4$ (if  $x_1=x_2$) or $xy_1x_1\alpha_1x_2y_2x$ forms a $C_6$ (if $x_1\ne x_2$) in $G$, a contradiction. Hence, $i=5.$
		\end{proof}

		
		
		We now divide $L^2 \setminus \mathcal{Q}$ into 3 disjoint subsets $A_0$, $A_1$ and $A_{\ge 2}$, where $A_i = \{x \in L^2 \setminus \mathcal{Q}:|L^3_F(x)| = i\}$ for $i = 0$ or 1, and $A_{\ge 2} = L^2 \setminus (\mathcal{Q} \cup A_0 \cup A_1) = \{x \in L^2 \setminus \mathcal{Q}:|L^3_F(x)| \ge 2\}$. Furthermore, let $A_{\ge 2} =\{w_1, w_2, \dots, w_\ell\}$.  

		The procedure for constructing $V^{\star}$ is as follows:	
	
		\textbf{Step 1:} Initialization: $V_0 = L^0 \cup L^1$;
		
		\textbf{Step 2:} Add all vertices of $L^2$ and their pendant vertices (if any $v \in L^2 \cap \mathcal{Q}$) into $V_0$. The resulting set is denoted by $V_1$;
		
		\textbf{Step 3:} For each vertex $x \in L^2$, add all vertices of $L^3_P(x)$ and their pendant vertices (if any $v \in L^3_P(x) \cap \mathcal{Q}$) into $V_1$. The resulting set is denoted by $V_2$;
		 
		\textbf{Step 4:} For each vertex $x \in A_1$, add the unique vertex of $L^3_F(x)$ into $V_2$. The resulting set is denoted by $V_3$;
		
		\textbf{Step 5:} For each vertex $x \in A_{\ge 2}$, add arbitrarily one vertex of $L^3_F(x)$  into $V_3$. The resulting set is denoted by $V_4$;
		
		\textbf{Step 6:} Add all remaining vertices  of $L^3$ (which is a subset of  $\bigcup_{i=1}^\ell L^3_F(w_i)$) together with certain vertices from $L^4\cup L^5$  into $V_4$. The resulting set is denoted by $V_5$;
		
		\textbf{Step 7:} For each vertex $x \in L^3$, add all remaining vertices of $L^4_P(x)$
		and their pendant vertices  (if any $v \in L^4_P(x) \cap \mathcal{Q}$) into $V_5$. 
		 The final set is $V^{\star}$.
	
		After finishing Step 5, we have 
		\[
		|V_4| = |L^0| + |L^1| + |L^2| + |L^2 \cap \mathcal{Q}| + |L^3_P| + |L^3_P \cap \mathcal{Q}| + |A_1| + |A_{\ge 2}|,
		\]
		and
		\begin{equation}\label{eq4}
			\begin{split}
				e(V_4) =~  & e(L^0 \cup L^1 \cup L^2) + e(L^2, L^3_P) +e(L^3_P)+ |(L^2 \cup L^3) \cap \mathcal{Q}| + |A_1| + |A_{\ge 2}|\\
				= ~& 1 + \frac{3}{2}|L^2|  + \frac{3}{2}|L^3_P| + |L^2 \cap \mathcal{Q}|+ |L^3 \cap \mathcal{Q}| + |A_1| + |A_{\ge 2}|\\
				= ~ & \frac{5}{4}|V_4| - \frac{3}{2} + \frac{1}{4}(|L^2| - |L^2 \cap \mathcal{Q}| - |A_1| - |A_{\ge 2}|) + \frac{1}{4}|L^3_P| - \frac{1}{4}|L^3_P \cap \mathcal{Q}|\\
				= ~& \frac{5}{4}|V_4| -\frac{3}{2}+\frac{1}{4}(|A_0|+|L^3_P \setminus \mathcal{Q}|)\\
				\ge ~ & \frac{5}{4}|V_4| -\frac{3}{2}.
			\end{split}
		\end{equation}
	
		Observed that Step 6 is the most important step in the entire procedure. In the following we aim to construct a set $V_5$ satisfying Properties 1, 2 and 3. Before this, we need some results and definitions.
	
		Choose an arbitrary vertex $w_1 \in A_{\ge 2}$, suppose that  $L^3_F(w_1) = \{x_1, x_2, \dots , x_t\}$ ($t \ge 2$) and $x_1$ has already been added into $V_4$ in Step~5. By Proposition~\ref{prop1}, for two  non-adjacent vertices $x_i, x_{i+1}$ in 
$L^3_F(w_1)$, three vertices $w_1,$ $x_i$ and $x_{i+1}$ lie on some induced $C_7$. Choose any such $C_7$ and denote it by $C_7(x_ix_{i+1}) = w_1x_iy_iz_iz_{i+1}y_{i+1}x_{i+1}w_1$. Moreover, we denote $V(C_7^-(x_i x_{i+1}))=V(C_7(x_ix_{i+1})) \setminus \{w_1, x_i\}$ and $E(C_7^-(x_i x_{i+1}))=E(C_7(x_ix_{i+1})) \setminus \{w_1x_i\}$ respectively.	

		The following five lemmas provide some results concerning the vertices in $V(C_7(x_ix_{i+1}))$. For readability, their proofs are deferred to the appendix.
	
		\begin{lem}\label{lem_yi_in_L4}
			$\{y_i, y_{i+1}\} \subseteq L^4$.
		\end{lem} 
	
		\begin{lem}\label{lem_yi_in_Q_implies_yi_in_L4p}
			If $y_i\in \mathcal{Q}$ $($resp. $y_{i+1}\in \mathcal{Q}),$  then $y_i \in L^4_P(x_i)$ $($resp. $y_{i+1} \in L^4_P(x_{i+1}))$ and there is a unique $y_i^\prime \in N_G(y_i) \cap L^4(x_i)$ $($resp.  $y_{i+1}^\prime \in N_G(y_{i+1}) \cap L^4(x_{i+1})).$
		\end{lem} 
		
		\begin{lem}\label{lem_zi_in_L45}
			$\{z_i, z_{i+1}\} \subseteq L^4 \cup L^5$.
		\end{lem}
		
		\begin{lem}\label{lem_zi_in_Q_implies_zi_in_L4}
			If $z_i\in  \mathcal{Q}$ $($resp. $z_{i+1}\in \mathcal{Q})$, then $z_i \in L^4(x_i^\prime)$ $($resp. $z_{i+1} \in L^4(x_{i+1}^\prime))$, where $x_i^\prime, x_{i+1}^\prime \in L^3 \setminus \{x_i, x_{i+1}\}.$ 
		\end{lem}
		
		\begin{lem}\label{lem_distinct}
			Suppose  $y_i^\prime, y_{i+1}^\prime, x_i^\prime, x_{i+1}^\prime$ are the vertices mentioned in Lemmas \ref{lem_yi_in_Q_implies_yi_in_L4p} and \ref{lem_zi_in_Q_implies_zi_in_L4}. Except for the possible $x_i^\prime = x_{i+1}^\prime$, the vertices $x_i, x_{i+1}, y_i, y_{i+1}, z_i, z_{i+1}, x_i^\prime,$ $x_{i+1}^\prime, y_i^\prime, y_{i+1}^\prime$ are pairwise distinct.		
		\end{lem}

		By considering the relationship between $\{z_i, z_{i+1}, y_i, y_{i+1}\}$ and $\mathcal{Q}$ together with Lemma \ref{lem_distinct}, there are 20 cases in total, which are summarized in the Table \ref{table_classification}, and their corresponding local structures are shown in Figures \ref{figcases1143} and \ref{figcases4473}.
		\begin{figure}[htbp]
			\centering
			\includegraphics[scale=0.85]{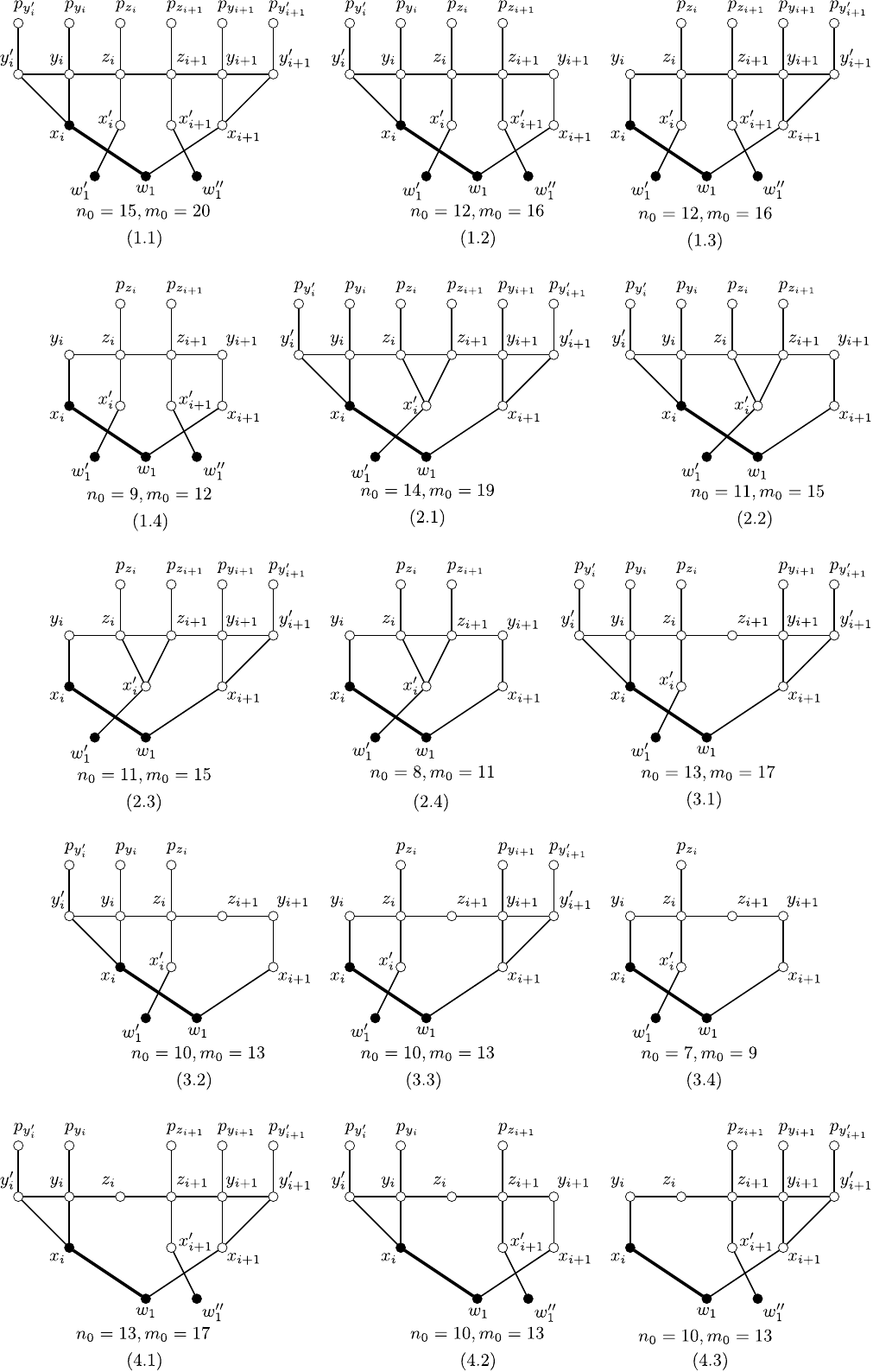}
			\caption{Local structures for Cases 1.1 -- 4.3}
			\label{figcases1143}
		\end{figure}
		\begin{figure}[htbp]
			\centering
			\includegraphics[scale=0.85]{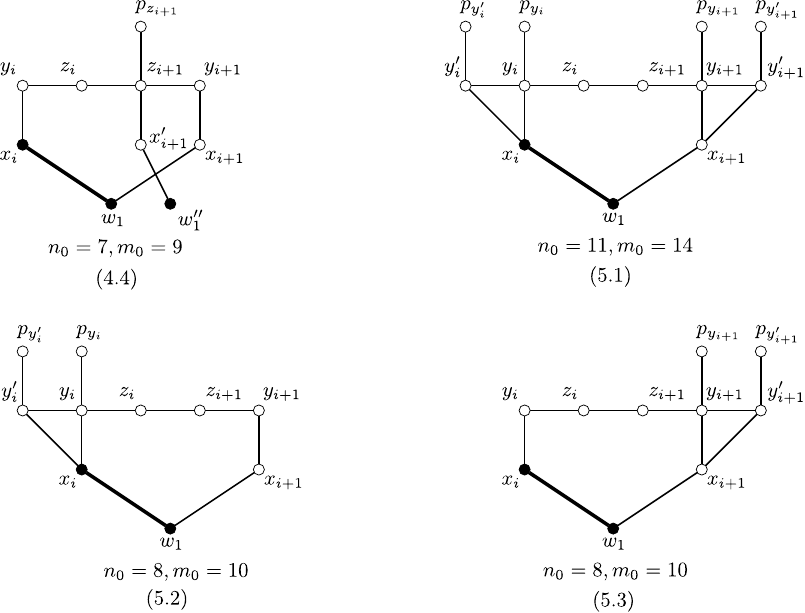}
			\caption{Local structures for  Cases 4.4 -- 5.3}
			\label{figcases4473}
		\end{figure}

		\begin{remark}
			The vertices $y_i^\prime$ and $y_{i+1}^\prime$ in  Figures~\ref{figcases1143} and \ref{figcases4473}  do not necessarily lie in $\mathcal{Q}$; hence, the vertices $p_{y_i^\prime}$ and $p_{y_{i+1}^\prime}$ need not exist. Such vertices $p_{y_i^\prime}$ and $p_{y_{i+1}^\prime}$ are called ``removable vertices''. Moreover, the vertices  $p_{u_4^\prime},$ $p_{v^\prime}$ and $p_{v_2}$ in Figure~\ref{figcases741745} are also removable vertices.		
		\end{remark}

		In each graph of all remaining figures, we adopt the following:
		\begin{enumerate}[itemsep=0pt, parsep=0pt]
			\item Black vertices and thick edges: already added in previous steps;
			 \item  White vertices and thin edges: undetermined;
			\item A thin edge is already added if and only if both its endpoints are already added.
		\end{enumerate}

		\begin{table}
			\caption{Classification of the 20 cases.}
			\label{table_classification}
			\centering
			\fontsize{10}{11}\selectfont
			\begin{tabular}
				{@{}ccccccc@{}}
				\toprule 
				\textbf{ }& $\{y_{i}, y_{i+1}\} \subseteq \mathcal{Q}$&only $y_i \in \mathcal{Q}$&only $y_{i+1} \in \mathcal{Q}$&$\{y_{i}, y_{i+1}\} \subseteq V \setminus \mathcal{Q}$ \\ 
				
				$\{z_{i}, z_{i+1}\} \subseteq \mathcal{Q} \text{ and } x_{i+1}^\prime \ne x_i^\prime$ & Case 1.1 & Case 1.2 & Case 1.3 & Case 1.4\\
				
				$\{z_{i}, z_{i+1}\} \subseteq \mathcal{Q}$ and $x_{i+1}^\prime = x_i^\prime$ & Case 2.1 & Case 2.2 & Case 2.3 & Case 2.4\\ 
				
				only $z_i \in \mathcal{Q}$ & Case 3.1 & Case 3.2 & Case 3.3 & Case 3.4\\
				
				only $z_{i+1} \in \mathcal{Q}$ & Case 4.1 & Case 4.2 & Case 4.3 & Case 4.4\\
				
				$\{z_{i}, z_{i+1}\} \subseteq V \setminus \mathcal{Q}$ & Case 5.1 & Case 5.2 & Case 5.3 & Case 5.4\\
				\hline
			\end{tabular}
		\end{table}
 
		Next, we  need to define some  vertex sets and edge sets: $U_i^\ast$,  $E_i^\ast,$  $U_i$,  $E_i,$  $U_i^-$ and $E_i^-$.

		\begin{myDef}
			For any graph in Figures~\ref{figcases1143}, \ref{figcases4473} and \ref{figcases741745}, define  $U_i^\ast$ as the set of all white vertices and $E_i^\ast$ as the set of all thin edges. Denote $n_0=|U_i^\ast|$ and $m_0=|E_i^\ast|$. See Tables \ref{table_results_Cases1173} and \ref{table_results_case743}.
		\end{myDef}
		Note that the removable vertices $p_{y_i^\prime},$ $p_{y_{i+1}^\prime},$ $p_{u_4^\prime}$, $p_{v^\prime}$ and $p_{v_2}$ may not exist.
		
		\begin{myDef}\label{def3}
Let $\mathcal{S}$ be the set of vertices $x$ in the collection
$\{y_i', y_{i+1}', v', u_4', v_2\}$ such that $x$ exists and $x \notin \mathcal{Q}$.
For each $x \in \mathcal{S}$, the vertex $p_x$ is   removable.
We now define
\[
U_i = U_i^* \setminus \{ p_x : x \in \mathcal{S} \}, \qquad
E_i = E_i^* \setminus \{ x p_x : x \in \mathcal{S} \}.
\]
\end{myDef}

	Since $U_i\subseteq U_i^\ast,$ then all  vertices in $U_i$ also have an undetermined status.
			
		\begin{myDef}\label{def2}
			Define $U_i^- \subseteq U_i$ and $E_i^- \subseteq E_i$ such that every vertex in $U_i^-$ and every edge in $E_i^-$ have already been added in previous steps, respectively.
		\end{myDef}

		Based on the above, we proceed to Step 6, which will be divided into two substeps.
		
		At the beginning, we set $V_4^1=V_4.$	

		\textbf{Step 6.1.}  Add all vertices of $U_i^+$  (see Remark \ref{r2})  into $V_4^{i}$ for $i=1,\dots,t-1.$ 
	 	The resulting set is denoted by $V_4^{i+1}=V_4^{i} \cup U_i^+$ (the last set $V_4^t$ contains all vertices of $L^3_F(w_1)$);
		
		\textbf{Step 6.2.} For any $w_j\in A_{\ge 2}$, where $j\ge 2$, add all remaining vertices of $L^3_F(w_j)$ together with certain vertices from $L^4\cup L^5$  into $V_4^t$. The resulting set is  $V_5$.
	
		In Step 6.1, we define the set of newly added edges as $E_i^+$ for $i=1,\dots,t-1,$ and then $E_i^+= E(U_i^+) \cup E(U_i^+, V_4^i).$ With this notation, we  have
	 	$e(V_4^{i+1}) \ge e(V_4^i) + |E_i^+|$. 
	
		\begin{remark}\label{r2}
			We shall explain the vertex set $U_i^+$ in Step 6.1: in  most cases, the vertex set $U_i^+=U_i\setminus U_i^-$, and then  $E_i^+=E_i\setminus E_i^-$. A precise explanation of $U_i^+$ will be given in the corresponding case.
		\end{remark}	
			
		Let $\rho_i=\frac{|E_i^+|}{|U_i^+|}$. The following lemma is useful and can be obtained by a simple calculation.	
		\begin{lem}\label{lemma3.7}
			If $e(V_4^{i})\ge \frac{5}{4} |V_4^{i}|-\ell$ and $\rho_i\ge \frac{5}{4}$ for $i\in \{ 1,\dots,t-1\},$ then $e(V_4^{i+1})\ge \frac{5}{4} |V_4^{i+1}|-\ell.$
		\end{lem}
		\begin{remark}\label{remark3}
			In our entire procedure, if the vertex $v\in \mathcal{Q}$ has already been added in a previous step, then its pendant vertex $p_v$ must certainly have been added as well.
		\end{remark}

	\paragraph{Step 6.1} 	
		We now proceed by induction on $t$ to prove that $e(V_4^{t})\ge \frac{5}{4} |V_4^{t}|-\frac{7}{4}$. For $t=1$, then $V_4^{1}=V_4$ and result holds by Eq. (\ref{eq4}).
		By the inductive hypothesis, assume that $e(V_4^{i})\ge \frac{5}{4} |V_4^{i}|-\frac{7}{4}$
		holds for $i= 1,\dots,t-2$. Next, we are going to construct suitable $U_i^+$ and $E_i^+$ such that $e(V_4^{i+1})\ge \frac{5}{4} |V_4^{i+1}|-\frac{7}{4}$. 
		In particular, if $x_{i+1} \in V_4^i$, we set $U_i^+ = E_i^+ = \emptyset$; then $V_4^{i+1} = V_4^i$, and the result follows by induction. Henceforth, we assume $x_{i+1} \notin V_4^i$, in which case $U_i^+$ contains the vertex $x_{i+1}$. Recall that $G+x_i x_{i+1}$ contains a $C_7(x_i x_{i+1})=w_1 x_i y_i z_i z_{i+1} y_{i+1} x_{i+1} w_1,$ and there are 20 cases concerning the relationship between $\{z_i, z_{i+1}, y_i, y_{i+1}\}$ and $\mathcal{Q}$ (see Table \ref{table_classification}). 
		We divide these $20$ cases into the following five parts:
\begin{itemize}[itemsep=0pt, parsep=0pt,leftmargin=3.5em]
\item[Part 1:] Cases 1.1 -- 5.3;
\item[Part 2:] Case 5.4 with $V(C_7^-(x_ix_{i+1})) \cap V_4^i \ne \emptyset$;
\item[Part 3:] Case 5.4 with $V(C_7^-(x_ix_{i+1})) \cap V_4^i = \emptyset$ and $\{z_i, z_{i+1}\} \cap L^4 \ne \emptyset$;
\item[Part 4:] Case 5.4 with $V(C_7^-(x_ix_{i+1})) \cap V_4^i = \emptyset$, $\{z_i, z_{i+1}\} \subseteq L^5$, and there exists a $C_7(x_{i_0}x_{i_0+1})$ for one $i_0 < i$ such that $\{z_{i_0}, z_{i_0+1}\} \subseteq L^5$;
\item[Part 5:] Case 5.4 with $V(C_7^-(x_ix_{i+1})) \cap V_4^i = \emptyset$, $\{z_i, z_{i+1}\} \subseteq L^5$, and for every $i_0 < i$, no $C_7(x_{i_0}x_{i_0+1})$ satisfies $\{z_{i_0}, z_{i_0+1}\} \subseteq L^5$.
\end{itemize}

		For Parts 1--4, we prove that there are suitable $U_i^+$ and $E_i^+$ such that $\rho_i \ge \frac{5}{4}$ (see Lemmas \ref{5.3.1} and \ref{B}). Thus, by the induction hypothesis and Lemma~\ref{lemma3.7}, we obtain $e(V_4^{i+1})\ge \frac{5}{4}|V_4^{i+1}| - \frac{7}{4}$.

		Firstly, we consider the Part 2 and Part 3.
		\begin{lem}\label{5.3.1}
			For Part 2 and Part 3, there are suitable $U_i^+$ and $E_i^+$ such that $\rho_i \ge \frac{5}{4}$.
		\end{lem}
		\begin{proof}
			For Part 2, we set 
			\[
			U_i^+= V(C_7^-(x_ix_{i+1})) \setminus V_4^i~~ \text{and}~~E_i^+ = E(C_7^-(x_ix_{i+1}))\setminus E_i^-.
			\]
			It is not hard to see that  $\rho_i \ge \frac{5}{4}$.
			
			For Part 3, without loss of generality, assume that $z_{i+1} \in L^4(x_{i+1}^\prime)$ and $x_{i+1}^\prime \in L^3(w_1^\prime)$, see Figure \ref{figcase74}, we set
			\[
			\begin{cases}
				U_i^+ =V(C_7^-(x_ix_{i+1})) \cup \{x_{i+1}^\prime\}~\text{and}~E_i^+ =E(C_7^-(x_ix_{i+1})) \cup \{z_{i+1}x_{i+1}^\prime, x_{i+1}^\prime w_1^\prime\}, &\text{if}~  x_{i+1}^\prime \notin V_4^i;\\
				U_i^+ =V(C_7^-(x_ix_{i+1})) ~\text{and}~E_i^+ =E(C_7^-(x_ix_{i+1})) \cup \{z_{i+1}x_{i+1}^\prime\}, &\text{if}~  x_{i+1}^\prime \in V_4^i.
			\end{cases}
			\]
			Then $\rho_i = \frac{8}{6}$ or  $\frac{7}{5}$, which implies that $\rho_i > \frac{5}{4}$.
		\end{proof}
	\begin{figure}
			\centering
			\includegraphics[scale=0.8]{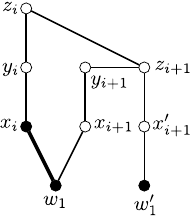}
			\caption{Case 5.4 with $\{z_i, z_{i+1}\} \cap L^4 \ne \emptyset$}
			\label{figcase74}
		\end{figure}
		Secondly, we  consider the Part 1 and Part 4.
		
		For Part 4, there exists a $C_7(x_{i_0}x_{i_0+1})$ for $i_0< i$ such that $\{z_{i_0}, z_{{i_0}+1}\} \subseteq L^5$. See Figure \ref{figcase743_c70}. We continue by classifying Part 4 with $z_{i_0} \not\sim z_{i+1}$. If $z_{i_0} \not\sim z_{i+1}$, by Fact 5, then there exists a path $P_{z_{i_0}z_{i+1}}$ of length 3, 4 or 5 connecting $z_{i_0}$ and $z_{i+1}$. Then we have 	
		\begin{claim}\label{c2}
			There exists a vertex $v$ in $V(P_{z_{i_0}z_{i+1}})\setminus (V(C_7^-(x_ix_{i+1}))\cup \{z_{i_0},z_{i+1}\})$
		 adjacent to some vertices of $\{x_{i+1}, y_i, y_{i+1}, z_i, z_{i+1}\}.$
		\end{claim}
		\begin{proof}
			If $V(P_{z_{i_0}z_{i+1}})\setminus \{z_{i_0},z_{i+1}\}\subseteq V(C_7^-(x_ix_{i+1}))$, then $z_{i_0}\sim y_i,$ but 
			$z_{i_0}y_ix_iw_1x_{i_0}y_{i_0}z_{i_0}$ forms a $C_6$  in $G,$ a contradiction. It follows that $V(P_{z_{i_0}z_{i+1}})\setminus (V(C_7^-(x_ix_{i+1}))\cup \{z_{i_0},z_{i+1}\})\ne \emptyset$ and there exists a vertex $v$ in $V(P_{z_{i_0}z_{i+1}})\setminus (V(C_7^-(x_ix_{i+1}))\cup \{z_{i_0},z_{i+1}\})$  adjacent to some vertices of $\{x_i,x_{i+1}, y_i, y_{i+1}, z_i, z_{i+1}\}$. If $v$ is only adjacent to $x_i,$ then the path $P_{z_{i_0}z_{i+1}}=z_{i_0}vx_iy_iz_iz_{i+1},$ 
			but 
			$z_{i_0}vx_iw_1x_{i_0}y_{i_0}z_{i_0}$ forms a $C_6$  in $G,$ a contradiction.	The completes the proof.
		\end{proof}	
		We still need to note that the following result holds.

		\begin{remark}\label{re3}
			For any vertex $v\in L^j$, $j\in \{4,5\},$ there is a path $P_{vu_2}=vu_{i}u_{i-1}\dots u_2$ connecting $v$ and $u_2,$ where $u_i\in L^i$ for $i\in \{2,3,4\}$.
		\end{remark}	
	
		According to the Claim \ref{c2}, we divide the Part 4 with $z_{i_0} \not\sim z_{i+1}$ into five cases.
		The local structures for each case are presented in Figure \ref{figcases741745}. 
		By Lemma \ref{lem_yi_in_Q_implies_yi_in_L4p}, if any  vertex $u_4\in L^4\cap \mathcal{Q}$ and $u_4\in L^4(u_3),$ then $u_4 \in L^4_P(u_3)$ and there is a unique $u_4^\prime \in N_G(u_4) \cap L^4(u_3).$
				
		\begin{figure}
			\centering
			\includegraphics[scale=0.9]{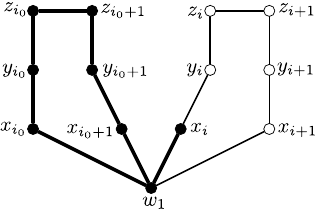}
			\caption {$C_7(x_{i_0}x_{i_0+1})\cup C_7(x_{i}x_{i+1})$ ($x_{i_0+1}$ and $x_{i}$ may  coincide)}
			\label{figcase743_c70}
		\end{figure}
	\begin{figure}
			\centering
			\includegraphics[scale=0.8]{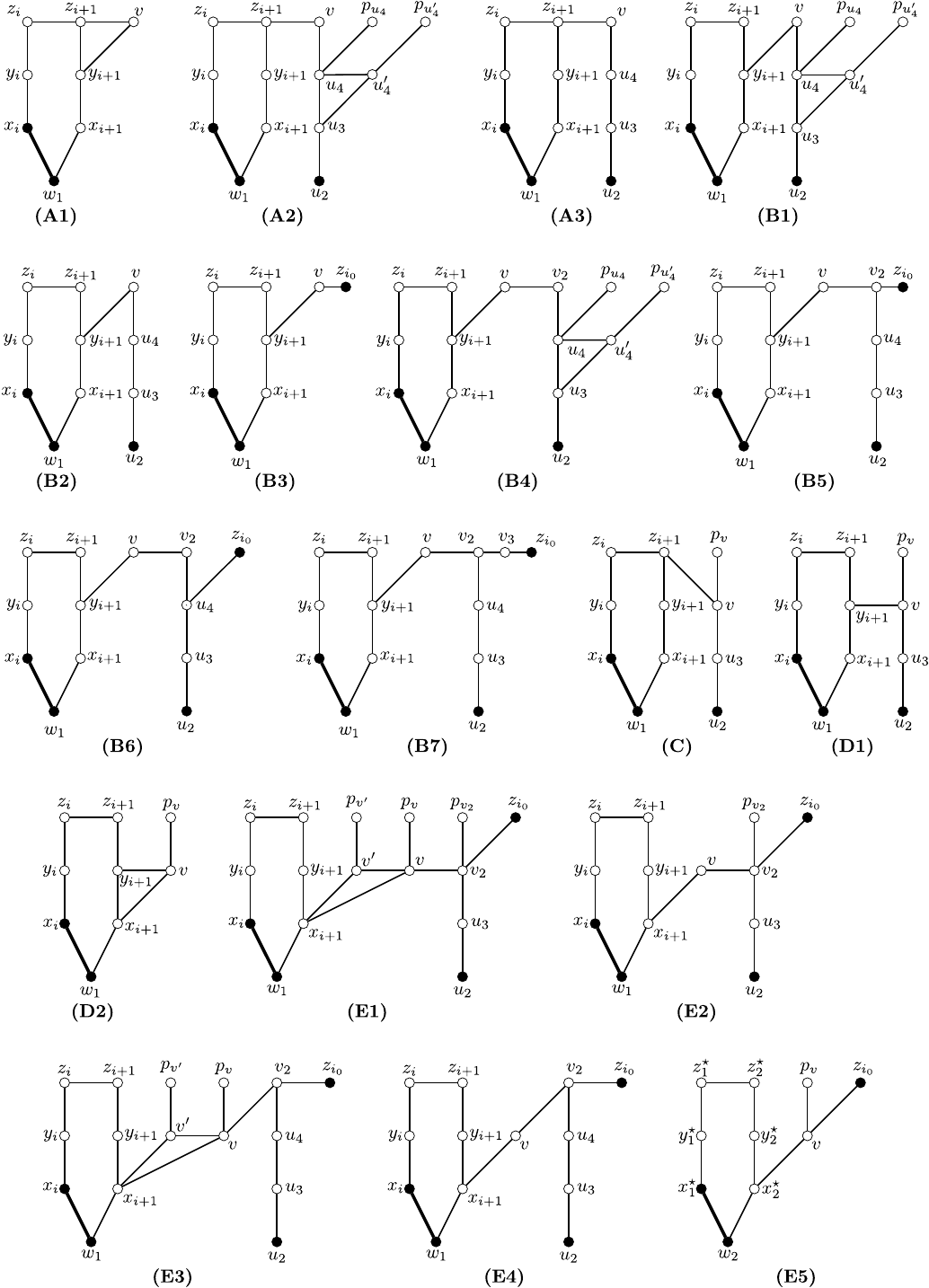}
			\caption{Local structures for each case of the Part 4 with $z_{i_0} \not\sim z_{i+1}$}
			\label{figcases741745}
		\end{figure}
		\textbf{Case 1.} $\{z_{i}, z_{i+1}\} \cap N_G(v) \ne \emptyset$ and $v \in L^5$. Without loss of generality, assume that $v\sim z_{i+1}.$ By Remark \ref{re3},  there exists a path  $P_{vu_2} = v u_4 u_3 u_2$. Then one of followings holds:
	 $u_4 = y_{i+1}$;
 $u_4 \ne y_{i+1}$ and $u_4 \in \mathcal{Q}$;   $u_4 \ne y_{i+1}$ and $u_4 \notin \mathcal{Q}$. See $\textbf{(A1)}$--$\textbf{(A3)}$ in Figure \ref{figcases741745}.
	
		\textbf{Case 2.} $\{y_i, y_{i+1}\} \cap N_G(v) \ne \emptyset$ and $v \in L^5$. Without loss of generality, assume that $v\sim y_{i+1}$. 
		
		If $|N_G(v) \cap L^4| \ge 2,$ by Remark \ref{re3}, there exists a path $P_{vu_2} = v u_4 u_3 u_2$, and one of followings holds:
	  $u_4\in \mathcal{Q};$
		 $u_4 \notin \mathcal{Q}$. See $\textbf{(B1)}$--$\textbf{(B2)}$ in Figure \ref{figcases741745}.

		If $|N_G(v) \cap L^4|=1$, since $d_G(v)\ge 2,$ then $v$ has a neighbor in $L^5,$ say $v_2$. By Remark~\ref{re3}, there exists a path $P_{v_2u_2} = v_2 u_4 u_3 u_2$.
		Recall that $P_{z_{i_0}z_{i+1}}=P_{z_{i_0}v}\cup P_{vz_{i+1}}$. 
		Since the length of $P_{vz_{i+1}}$ is 2,  
		then $P_{z_{i_0}v}$ has length 1, 2 or 3. 
			Let $r^\prime$ denote the length of  $P_{z_{i_0}v}.$
			If $r^\prime=1,$ then $v_2=z_{i_0};$ if $r^\prime=2,$ 
we write $P_{z_{i_0}v}=z_{i_0}v_2v$; if $r^\prime=3,$ we write $P_{z_{i_0}v}=z_{i_0}v_2v_3v$. Then one of followings holds:
		 $r^\prime=1;$ $r^\prime\ge 2$ and $u_4 \in \mathcal{Q}$; 
		 $r^\prime=2$ and $u_4 \not\in \mathcal{Q}$;
		$r^\prime=3,$ $u_4 \not\in \mathcal{Q},$ $v_3\in L^4$ and $v_3=u_4$;
		$r^\prime=3,$ $u_4 \not\in \mathcal{Q},$ and $v_3\in L^5$.
                 See $\textbf{(B3)}$--$\textbf{(B7)}$ in Figure \ref{figcases741745}.

		\textbf{Case 3.} $\{z_{i}, z_{i+1}\} \cap N_G(v) \ne \emptyset$ and $v \in L^4$. Without loss of generality, assume that $v\sim z_{i+1}$. By Remark \ref{re3}, there is a path $P_{vu_2} = v u_3 u_2$. 
		                 See $\textbf{(C)}$ in Figure \ref{figcases741745}.
	
		\textbf{Case 4.} $\{y_i, y_{i+1}\} \cap N_G(v) \ne \emptyset$ and $v \in L^4$. Without loss of generality, assume that $v\sim y_{i+1}.$ By Remark~\ref{re3}, there is a path $P_{vu_2} = v u_3 u_2$.	Then one of followings holds: 
$u_3 \ne x_{i+1}$;  $u_3 = x_{i+1}$.                  See $\textbf{(D1)}$--$\textbf{(D2)}$ in Figure \ref{figcases741745}.

		\textbf{Case 5.} $v\sim x_{i+1}.$ 			
 Then $v \in L^4$ and the path $P_{z_{i_0}z_{i+1}}$ has length at least 4. Let $r$ denote the length of  $P_{z_{i_0}z_{i+1}}.$ If  $r=4,$ then $P_{z_{i_0}z_{i+1}}=z_{i_0}vx_{i+1}y_{i+1}z_{i+1}$, but $z_{i_0}y_{i_0}x_{i_0}w_1x_{i+1}vz_{i_0}$ forms a $C_6$ in $G$, a contradiction. 
So, $r=5$ and write  $P_{z_{i_0}z_{i+1}}=z_{i_0}v_2vx_{i+1}y_{i+1}z_{i+1}.$ Then one of the followings holds:
 $v \in \mathcal{Q}$ and $v_2 \in L^4$; 
		 $v \notin \mathcal{Q}$ and $v_2 \in L^4$;
	 $v \in \mathcal{Q}$ and $v_2 \in L^5$;
 $v \notin \mathcal{Q}$ and $v_2 \in L^5$.    See $\textbf{(E1)}$--$\textbf{(E4)}$ in Figure \ref{figcases741745}.

			\begin{table}
			\caption{Results for Cases 1.1--5.3}
			\label{table_results_Cases1173}
			\centering
			\fontsize{9}{11}\selectfont
			\begin{tabular}
				{@{}cccccccc@{}}
				\toprule 
				\textbf{Case}& $n_0$ & $m_0$ & removable vertices & $a_{\max}$ & $a_{\min}$ & $(a, 4a + 5n_0 - 4m_0, k_{\min})$\\ 
				
				\textbf{1.1}  &  $15$ & $20$ & $p_{y_i^\prime}, p_{y_{i+1}^\prime}$ & $3$ & $2$ & $(3,7,10),(2,3,7)$\\
				
				\textbf{1.2}  &  $12$ & $16$ & $p_{y_i^\prime}$ & $3$ & $2$ & $(3,8,10),(2,4,7)$\\ 
				
				\textbf{1.3}  &  $12$ & $16$ & $p_{y_{i+1}^\prime}$ & $2$ & $2$ & $(2,4,7)$\\
				
				\textbf{1.4}  &  $9$ & $12$ & $\emptyset$ & $2$ & $1$ & $(2,5,7),(1,1,4)$\\
				
				\textbf{2.1}  &  $14$ & $19$ & $p_{y_i^\prime}, p_{y_{i+1}^\prime}$ & $3$ & $2$ & $(3,6,9),(2,2,7)$\\
				
				\textbf{2.2}  &  $11$ & $15$ & $p_{y_i^\prime}$ & $3$ & $2$ & $(3,7,9),(2,3,7)$\\ 
				
				\textbf{2.3}  &  $11$ & $15$ & $p_{y_{i+1}^\prime}$ & $2$ & $2$ & $(2,3,6)$\\ 
				
				\textbf{2.4}  &  $8$ & $11$ & $\emptyset$ & $2$ & $2$ & $(2,4,6)$\\
				
				\textbf{3.1}  &  $13$ & $17$ & $p_{y_i^\prime}, p_{y_{i+1}^\prime}$ & $2$ & $1$ & $(2,5,7),(1,1,4)$\\
				
				\textbf{3.2}  &  $10$ & $13$ & $p_{y_i^\prime}$ & $2$ & $1$ & $(2,6,7),(1,2,4)$\\ 
				
				\textbf{3.3}  &  $10$ & $13$ & $p_{y_{i+1}^\prime}$ & $1$ & $1$ & $(1,2,4)$\\ 
				
				\textbf{3.4}  &  $7$ & $9$ & $\emptyset$ & $1$ & $1$ & $(1,3,4)$\\
				
				\textbf{4.1}  &  $13$ & $17$ & $p_{y_i^\prime}, p_{y_{i+1}^\prime}$ & $2$ & $1$ & $(2,5,8),(1,1,4)$\\
				
				\textbf{4.2}  &  $10$ & $13$ & $p_{y_i^\prime}$ & $2$ & $1$ & $(2,6,8),(1,2,4)$\\ 
				
				\textbf{4.3}  &  $10$ & $13$ & $p_{y_{i+1}^\prime}$ & $1$ & $1$ & $(1,2,4)$\\ 
				
				\textbf{4.4}  &  $7$ & $9$ & $\emptyset$ & $1$ & $1$ & $(1,3,5)$\\
				
				\textbf{5.1}  &  $11$ & $14$ & $p_{y_i^\prime}, p_{y_{i+1}^\prime}$ & $1$ & $1$ & $(1,3,4)$\\ 
				
				\textbf{5.2}  &  $8$ & $10$ & $p_{y_i^\prime}$ & $1$ & $1$ & $(1,4,4)$\\
				
				\textbf{5.3}  &  $8$ & $10$ & $p_{y_{i+1}^\prime}$ & $0$ & $1$ & $-$\\
				\hline
			\end{tabular}
		\end{table}

		Based on the above classification, we have the following lemma.

		\begin{lem}\label{B}
			For Part 1 and Part 4, there are suitable $U_i^+$ and $E_i^+$ such that $\rho_i \ge \frac{5}{4}$. 
		\end{lem}
		\begin{proof}
			For Part 4 with $z_{i_0}\sim z_{i+1}$, we set 
			\[
			U_i^+= V(C_7^-(x_ix_{i+1}))~~ \text{and}~~E_i^+ = E(C_7^-(x_ix_{i+1}))\cup \{z_{i_0}z_{i+1}\}.
			\]
			Then $\rho_i = \frac{7}{5} > \frac{5}{4}$.
				
			For Part 1 and  Part 4 with $z_{i_0}\not\sim z_{i+1}$, the vertex set $U_i^+=U_i\setminus U_i^-$ and the edge set $E_i^+=E_i\setminus E_i^-$. As $x_{i+1}\in U_i^+,$ then $U_i^+$ is nonempty. Recall that 
			$U_i = U_i^* \setminus \{ p_x : x \in \mathcal{S} \}$ (see Definition \ref{def3}).
			We first assume that $U_i=U_i^\ast,$ that is, $|U_i|=n_0$ and $|E_i|=m_0$. Let $|U_i^-| = k$ and $|E_i^-| = k + a$, where $k \ge 0$.
			We aim to prove that, for any nonempty $U_i^+ = U_i \setminus U_i^-$, the following  holds:
		\begin{equation}\label{Ineq1}
			\rho_i=\frac{|E_i^+|}{|U_i^+|} = \frac{m_0 - (k+a)}{n_0 - k} \ge \frac{5}{4},\quad\text{or equivalently,}\quad k \ge 4a + 5n_0 - 4m_0.
		\end{equation}
		If $4a + 5n_0 - 4m_0 \le 0$,  Eq. \eqref{Ineq1} holds for every $k \ge 0$. It suffices to consider that $4a + 5n_0 - 4m_0 \ge 1$. Let $a_{\min}$ be the least integer $a$ such that $4a + 5n_0 - 4m_0 \ge 1$. For each case, the value of $a_{\min}$ is displayed in the sixth column of Tables \ref{table_results_Cases1173} and \ref{table_results_case743}. Next, we establish an upper bound for $a$, denoted by $a_{\max}$. For each graph in Figures \ref{figcases1143}, \ref{figcases4473} and \ref{figcases741745}, we define a graph $H$ by contracting all the black vertices into a single vertex, say $b_0$, and deleting the vertex $x_{i+1}$. For example, for Case 1.1 and Case 2: \textbf{(B6)}, the graph $H$ is shown in Figure \ref{fig_H}.	
		
		Let $H_i=H[U_i^- \cup \{b_0\}],$ then	
		$V(H_i)=U_i^- \cup \{b_0\}$ and $E(H_i)=E_i^-.$
	
				\begin{figure}
		\centering
		\includegraphics[width=13cm,height=3.5cm]{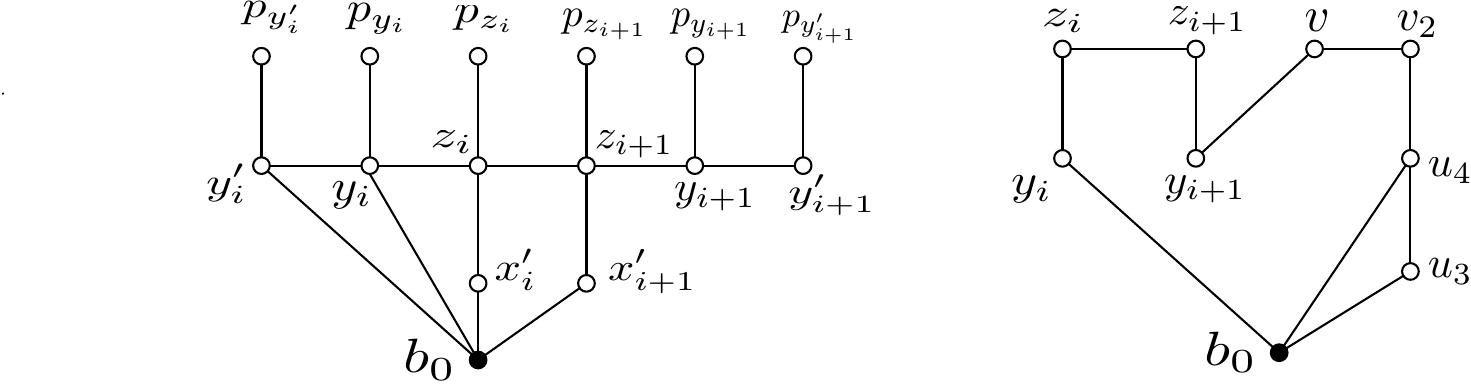}
		\caption{Graph $H$ for Case 1.1 and Case 2: \textbf{(B6)}}
		\label{fig_H}
	\end{figure}

		\begin{claim}\label{a} $a= c_1 - c_2 + 1,$ where 
			$c_1$ and $c_2$ are respectively the cyclomatic number and the number of connected components of $H_i.$
		\end{claim}
	
		\begin{proof}
			It is well known that for any graph $H$, $|E(H)| - |V(H)| = c_1 - c_2$, where $c_1$ and $c_2$ are respectively the cyclomatic number and the number of connected components of $H$. Then $|E(H_i)| - |V(H_i)|=|E_i^-| - |U_i^- \cup \{b_0\}| = c_1- c_2,$ and $a = |E_i^-| - |U_i^-| = c_1- c_2+ 1$.	
		\end{proof}
		Since $c_2\ge 1,$ then $a\le  c_1$. For each case, the cyclomatic number of $H$ is readily obtained, and hence an upper bound $a_{\max}$ for $a$ follows immediately, which is listed in the fifth column of Tables~\ref{table_results_Cases1173} and \ref{table_results_case743}.
		Notice that for each case of the Part 4 with $z_{i_0} \not\sim z_{i+1}$, since $V(C_7^-(x_ix_{i+1})) \cap V_4^i=\emptyset,$
		then $U_i^-$ contains no vertex of $V(C_7^-(x_ix_{i+1})).$
	For example, for Case 2: \textbf{(B6)}, since  $U_i^-$ contains no vertex of $\{y_i,y_{i+1},z_i,z_{i+1}\},$ then   
	$H_i$ has at most one cycle. It follows that $a_{\max}=1.$
	
	 If $a_{\max}<a_{\min},$ then Eq. (\ref{Ineq1}) holds trivially.
		We take Case 1.1  as an example to illustrate that  Eq. (\ref{Ineq1}) holds. 		
		For Case 1.1, from Table~\ref{table_results_Cases1173}, we have $a \in [2,3]$. Let $k_{\min}$ be the smallest cardinality of set  $U^-_{i}$. If $a=3$, then $4a + 5n_0 - 4m_0=7$ and $c_1=3$, which implies that $H_i$ contains three cycles: $b_0y_iy_i'b_0,$ $b_0y_iz_ix_i^\prime b_0$ and $b_0x_i^\prime  z_iz_{i+1}x_{i+1}^\prime b_0$. Now $U^-_{i}$ contains vertices $y_i, y_i', z_i, z_{i+1}, x_{i}^\prime, x_{i+1}'$. By Remark \ref{remark3}, $U^-_{i}$ also contains $p_{y_i}, p_{y_i'}, p_{z_i}, p_{z_{i+1}}$. Thus, $U^-_{i}$ contains at least ten vertices and $k\ge k_{min}\ge 10> 7=4a + 5n_0 - 4m_0,$ as desired. If $a=2,$ then $4a + 5n_0 - 4m_0=3$ and $c_1\ge 2,$ which implies that $H_i$ contains two cycles. It is easy to see that $U^-_{i}$ contains at least seven vertices. Thus, $k\ge k_{min}> 3=4a + 5n_0 - 4m_0,$  as desired. Hence, Eq. (\ref{Ineq1}) holds for Case 1.1.
	
		\begin{table}
			\caption{Results for  Part 4 with $z_{i_0} \not\sim z_{i+1}$}
			\label{table_results_case743}
			\centering
			\fontsize{9}{11}\selectfont
			\begin{tabular}
				{@{}ccccccc@{}}
				\toprule 
				& $n_0$ & $m_0$ & removable vertices & $a_{\max}$ & $a_{\min}$ & $(a, 4a + 5n_0 - 4m_0, k_{\min})$\\ 
				
				\textbf{(A1)}  &  $6$ & $8$ & $\emptyset$ & $0$ & $1$ & $-$\\
				
				\textbf{(A2)} &  $11$ & $14$ &$p_{u_4^\prime}$& $1$ & $1$ & $(1, 3, 5)$\\ 
				
				\textbf{(A3)} & $8$ & $10$ &$\emptyset$& $0$ & $1$ & $-$\\
				
				\textbf{(B1)}  &  $11$ & $14$ &$p_{u_4^\prime}$& $1$ & $1$ & $(1, 3, 5)$\\
				
				\textbf{(B2)}  &  $8$ & $10$ &$\emptyset$& $0$ & $1$ & $-$\\
				
				\textbf{(B3)}  &  $6$ & $8$ &$\emptyset$& $0$ & $1$ & $-$\\
				
				\textbf{(B4)}  &  $12$ & $15$ &$p_{u_4^\prime}$& $1$ & $1$ & $(1, 4, 5)$\\ 
				
				\textbf{(B5)}  & $9$ & $12$ &$\emptyset$& $1$ & $1$ & $(1, 1, 3)$\\
				
				\textbf{(B6)} & $9$ & $12$ &$\emptyset$& $1$ & $1$ & $(1, 1, 2)$\\
				
				\textbf{(B7)} & $10$ & $13$ &$\emptyset$& $1$ & $1$ & $(1, 2, 4)$\\
				
				\textbf{(C)}  &  $8$ & $10$ &$p_{v}$& $0$ & $1$ & $-$\\
				
				\textbf{(D1)} & $7$ & $9$ &$p_{v}$& $0$ & $1$ & $-$\\
				
				\textbf{(D2)} &  $8$ & $10$ &$p_{v}$& $0$ & $1$ & $-$\\ 
				
				\textbf{(E1)}  &  $12$ & $16$ &$p_{v^\prime}, p_{v_2}$& $1$ & $2$ & $-$\\
				
				\textbf{(E2)}  &  $9$ & $12$ &$p_{v_2}$& $1$ & $1$ & $(1, 1, 3)$\\ 
				
				\textbf{(E3)}  & $12$ & $16$ &$p_{v^\prime}$& $1$ & $2$ & $-$\\
				
				\textbf{(E4)} & $9$ & $12$ &$\emptyset$& $1$ & $1$ & $(1, 1, 3)$\\
				\hline
			\end{tabular}
		\end{table}
	
		For all remaining cases, the proof is similar, the last column of Tables \ref{table_results_Cases1173} and \ref{table_results_case743} yields that  $k_{\min} \ge 4a + 5n_0 - 4m_0$  holds for any $a\in [a_{min},a_{max}]$. Hence, Eq. (\ref{Ineq1}) holds.	

		Next we shall consider that
		$U_i = U_i^* \setminus \{ p_x : x \in \mathcal{S} \}\subset U_i^\ast,$ where $\mathcal{S}$ is nonempty.
		Similarly, we set $|U_i|=n_0^\prime,$ $|E_i|=m_0^\prime$, $|U_i^-| = k^\prime,$ $|E_i^-| = k^\prime+ a^\prime$, a graph $H^\prime$ by contracting all the black vertices  into a single vertex, say $b_0^\prime$, and deleting the vertex $x_{i+1}$, and $H_i^\prime=H^\prime[U_i^- \cup \{b_0^\prime\}].$ We still aim to prove that 	$k^\prime \ge 4a^\prime + 5n_0^\prime - 4m_0^\prime$ holds for  any $U_i^-\subset U_i$. Here we only consider that $U_i  = U_i^\ast \setminus\{p_{y_{i}^\prime}\}$; the others are similar. Then $n_0^\prime=n_0-1,$  $m_0^\prime=m_0-1$, $k^\prime=k-1$ if $y_{i}^\prime\in U_i^-$ and $k^\prime=k$ if $y_{i}^\prime\notin U_i^-$. One can see that there is a one-to-one correspondence between $H_i\setminus \{p_{y_{i}^\prime}\}$ and $H_i^\prime$. Since $p_{y_{i}^\prime}$ is the pendant vertex, by Claim \ref{a} and Remark \ref{remark3}, then $a=a^\prime$. By Eq. (\ref{Ineq1}), we have $k^\prime\ge k-1\ge 4a + 5n_0 - 4m_0-1=4a^\prime + 5n_0^\prime - 4m_0^\prime$, as desired.

		Consequently, $\rho_i \ge \frac{5}{4}$ holds. This completes the proof.
		\end{proof}

		Finally, we consider Part 5. If there is no $C_7(x_{i_0}x_{i_0+1})$ for all $i_0< i$ such that $\{z_{i_0}, z_{{i_0}+1}\} \subseteq L^5$, by Lemmas~\ref{5.3.1} and \ref{B}, we have $\rho_j\ge \frac{5}{4}$ for all $j=1,\dots,i-1$. By Eq. (\ref{eq4}), we have
		\[
		e(V_4^{i})=e(V_4^{1})+\sum_{j=1}^{i-1}|E_{j}| \ge \frac{5}{4}|V_4^{1}|-\frac{3}{2}+\sum_{j=1}^{i-1} 
		 \frac{5}{4}|U_{j}|= \frac{5}{4}(|V_4^{1}|+\sum_{j=1}^{i-1} 
		|U_{j}|)-\frac{3}{2}=\frac{5}{4}|V_4^{i}|-\frac{3}{2}.
		\]
		Now we set $U_i^+= V(C_7^-(x_ix_{i+1}))$ and $E_i^+ = E(C_7^-(x_ix_{i+1})),$ which implies that $|U_i^+|=5$ and $|E_i^+|=6$. 
		Hence,  $e(V_4^{i+1})=e(V_4^{i})+6 \ge \frac{5}{4}|V_4^{i}| - \frac{3}{2}+6=\frac{5}{4}(|V_4^{i}|+5) - \frac{7}{4}= \frac{5}{4}|V_4^{i+1}| - \frac{7}{4}$.
		 
		The step 6.1 is finished.

	\paragraph{Step 6.2.}
		In this step, we will add all vertices of $L^3_F(w_i)$  ($i=2,\dots,\ell$) into $V_4^t$. We now consider that $i=2$. Suppose that $L^3_F(w_2)=\{x_1^\star,x_2^\star,\dots,x_k^\star\}$ and $x_1^\star$ has been added in Step 5. By Proposition~\ref{prop1}, three vertices $w_2$, $x_1^\star$ and $x_{2}^\star$ lie on an induced $C_7(x_1^\star x_2^\star)=w_2 x_1^\star y_1^\star z_1^\star z_2^\star y_2^\star x_2^\star w_2$. Similar to Step 6.1, we want to construct $V_4^{t+1}$ from $V_4^t$ such that $V_4^{t+1} \setminus V_4^t$ contains the vertex $x_2^\star$; in particular, if $x_2^\star \in V_4^t$, then we set $V_4^{t+1} = V_4^t$. Thus, we  assume that $x_2^\star \notin V_4^t$. Let $U_{t+1}^+=V_4^{t+1} \setminus V_4^t$ and $E_{t+1}^+=E(U_{t+1}^+) \cup E(U_{t+1}^+,V_4^{t})$, where  $U_{t+1}^+$ contains the vertex $x_2^\star$.

		If $e(V_4^t) \ge \frac{5}{4}|V_4^t| - \frac{3}{2}$, analogously to Step 6.1, we can construct suitable $U_{t+1}^+$ and $E_{t+1}^+$ such that $e(V_4^{t+1}) \ge \frac{5}{4}|V_4^{t+1}| - \frac{7}{4}$. Otherwise, $e(V_4^t)=\frac{5}{4}|V_4^t| - \frac{7}{4}$. By Step~6.1, there exist some indexes $j \in \{1,\dots,t-1\}$ such that $e(V_4^j)=\frac{5}{4}|V_4^j|-\frac{3}{2}$ and
		$e(V_4^{j+1})=\frac{5}{4}|V_4^{j+1}|-\frac{7}{4}$. Let $j_0$ be the largest such index. Then 
		$U_{j_0}^+=V(C_7^-(x_{j_0} x_{j_0+1})),$ $E_{j_0}^+=E(C_7^-(x_{j_0} x_{j_0+1}))$ and
		$\{z_{j_0},z_{j_0+1}\}\subseteq L^5$.

		\begin{claim}\label{claim:property_of_Uj0}
			For any $i$ with $j_0+1\le i\le t$, $e(V_4^i) = \frac{5}{4}|V_4^i| - \frac{7}{4}$.
		\end{claim}
\begin{proof}
Let $d_i=e(V_4^i) -( \frac{5}{4}|V_4^i| - \frac{7}{4})$ for $j_0+1\le i\le t.$
For $i=j_0+2,\dots,t,$ we have 
$d_i-d_{i-1}=e(V_4^i) -( \frac{5}{4}|V_4^i| - \frac{7}{4})-(e(V_4^{i-1}) -( \frac{5}{4}|V_4^{i-1}| - \frac{7}{4}))=|E_i^+|-\frac{5}{4}|U_i^+|.$
Since $\rho_i\ge \frac{5}{4},$ then $d_i-d_{i-1}\ge 0.$
Note that $d_{j_0+1}=d_t=0,$ then $d_i=0$ for all $i=j_0+2,\dots,t.$ Hence, $e(V_4^i) = \frac{5}{4}|V_4^i| - \frac{7}{4}$.
\end{proof}

		By Claim~\ref{claim:property_of_Uj0}, for   any $i$ with $j_0+1\le i\le t,$ we have $\rho_i= \frac{5}{4}$.

 If $V(C_7^-(x_1^\star x_{2}^\star)) \cap V_4^t\ne \emptyset$, or 
		$V(C_7^-(x_1^\star x_{2}^\star)) \cap V_4^t=\emptyset$ with $\{z_1^\star, z_{2}^\star\} \cap  L^4\ne \emptyset,$ similar to the Parts 2 and 3 in Step 6.1, there exist suitable $U_{t+1}^+$ and $E_{t+1}^+$  such that $\rho_{t+1}\ge \frac{5}{4}$. Otherwise, $V(C_7^-(x_1^\star x_{2}^\star)) \cap V_4^t=\emptyset$ and 
		$\{z_1^\star, z_{2}^\star\} \subseteq  L^5$. If $z_{j_0}$ is adjacent to  one vertex $v$ in $\{y_1^\star, y_{2}^\star, z_1^\star, z_{2}^\star\},$ we set  
		\[
		U_{t+1}^+ =V(C_7^-(x_1^\star x_{2}^\star))~\text{and}~E_{t+1}^+ =E(C_7^-(x_1^\star x_{2}^\star)) \cup \{z_{j_0}v\}.
		\]
		Then $\rho_{t+1}=\frac{7}{5} > \frac{5}{4}$.
		If $z_{j_0} \nsim z_{2}^\star$, there exists a path $P_{z_{j_0}z_2^\star}$ of length 3, 4 or 5 connecting $z_{j_0}$ and $z_2^\star$.
		Similar to Claim \ref{c2}, there exists a vertex $v \in V(P_{z_{j_0}z_2^\star})\setminus(V(C_7^-(x_ix_{i+1})\cup \{z_{i_0}, z_2^\star\})$ adjacent to some vertices in $\{x_1^\star,x_2^\star, y_1^\star, y_2^\star, z_1^\star, z_2^\star\}$.
		If $v$ is adjacent to some vertices in $\{x_2^\star, y_1^\star, y_2^\star, z_1^\star, z_2^\star\}$, it can be handled analogously to Part~4 with $z_{i_0}\not\sim z_{i+1}$ in Step~6.1, yielding that $\rho_{t+1} \ge \frac{5}{4}$. The only  difference occurs in Case~5:  in Step~6.1 
		the length of path $P_{z_{i_0}z_{i+1}}$ must equal to  $5$, here the corresponding path $P_{z_{i_0}z_2^\star}$ may  have length $4$, namely $P_{z_{i_0}z_2^\star} = z_{i_0} v x_2^\star y_2^\star z_2^\star$, which gives one extra configuration \textbf{(E5)} (see Figure~\ref{figcases741745}). One can easily check that  the inequality $\rho_{t+1} \ge \frac{5}{4}$ still holds for   \textbf{(E5)}. Therefore, the only remaining case is when $v$ is adjacent to $x_1^\star$, which yields the path $P_{z_{j_0} z_2^\star} = z_{j_0} v x_1^\star y_1^\star z_1^\star z_2^\star$ (see Figure~\ref{figcasestep62}).
		\begin{figure}
			\centering
			\includegraphics[scale=0.8]{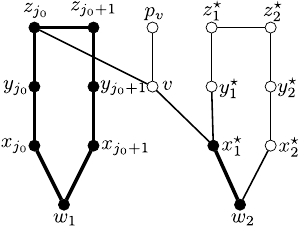}
			\caption{Step 6.2 with $v$ adjacent to $x_1^\star$}
			\label{figcasestep62}
		\end{figure}
		
		\begin{claim}\label{c4}
			$v \notin V_4^t.$
		\end{claim}
		\begin{proof}
			Suppose for contradiction that $v \in V_4^t$.
			Then $v$ belongs to one $U_{j_1}^+$.
			 If $j_1 \le j_0$, then during the process of adding $U_{j_0}^+$ to $V_4^{j_0}$, we should set
			$E(C_7^-(x_{j_0}x_{j_0+1})) \cup \{v z_{j_0}\} \subseteq E_{j_0}^+,$
			which yields $\rho_{j_0} \ge \frac{7}{5} > \frac{5}{4}$, contradicting  that
			$e(V_4^{j_0+1}) = \frac{5}{4}|V_4^{j_0+1}| - \frac{7}{4}.$
			So, $j_1 > j_0$.
			Since  $\rho_i= \frac{5}{4}$ for all $j_0+1\le i\le t,$  then $\rho_{j_1} = \frac{5}{4}$.
			By Eq. (\ref{Ineq1}), then $k=4a + 5n_0 - 4m_0.$ 
			 If $k=4a + 5n_0 - 4m_0\ge 1,$  by Tables~\ref{table_results_Cases1173} and~\ref{table_results_case743},
			then $k=k_{min}=4a + 5n_0 - 4m_0$ if and only if $U_{j_1}^+$  belongs to Case 5.2, 
			$U_{j_1}^+=U_{j_1}^*~\text{and} ~E_{j_1}^+=E_{j_1}^*,~\text{or}~U_{j_1}^+=\{z_{i},z_{i+1},y_{i+1},x_{i+1}\}~\text{and} ~E_{j_1}^+=\{
			y_iz_i,z_iz_{i+1},y_{i+1}x_{i+1},x_{i+1}w_1\};$
			if $k=4a + 5n_0 - 4m_0\le  0,$ then $k=0.$
			It follows that $a=0$ and  $\frac{m_0}{n_0}=\frac{5}{4}.$ By Tables~\ref{table_results_Cases1173} and~\ref{table_results_case743}, 
	                 $U_{j_1}^+$  belongs to one of the following cases: Cases 5.2, 5.3,   \textbf{(A3)}, \textbf{(B2)}, \textbf{(C)} and \textbf{(D2)}. Moreover, we have 
                         $U_{j_1}^+=U_{j_1}^*$ and $E_{j_1}^+=E_{j_1}^*.$

			If $z_{j_0} \notin U_{j_1}$, then $v z_{j_0}\not\in E_{j_1}^+.$
			Now we add the edge $v z_{j_0}$ into $E_{j_1}^+,$ obtaining
			 a larger edge set $E_{j_1}^+$ such that $\rho_{j_1} > \frac{5}{4}$, contradicting $\rho_{j_1} = \frac{5}{4}$.
			
			If $z_{j_0} \in U_{j_1}$, then $z_{j_0}\in U_{j_1}^-$ and $|U_{j_1}^-|=k>0.$ Then  $U_{j_1}^+$  belongs to Case 5.2, $U_{j_1}^+=\{z_{i},z_{i+1},y_{i+1},x_{i+1}\}$ and $E_{j_1}^+=\{y_iz_i,z_iz_{i+1},y_{i+1}x_{i+1},x_{i+1}w_1\}.$ In Case 5.2, no such vertex $z_{j_0}$ in $ U_{j_1}^-=U_{j_1} \setminus U_{j_1}^+$ belongs to $L^5,$	
			a contradiction.
	
	Hence, $v \notin V_4^t.$		
This completes the proof.		\end{proof}
		By Claim \ref{c4}, we set
		\[
		\begin{cases}
			U_{t+1}^+ =V(C_7^-(x_1^\star x_{2}^\star)) \cup \{v\}~\text{and}~E_{t+1}^+ =E(C_7^-(x_1^\star x_{2}^\star)) \cup \{z_{j_0}v, vx_{1}^\star\}, &~\text{if}~  v \notin \mathcal{Q};\\
			U_{t+1}^+ =V(C_7^-(x_1^\star x_{2}^\star)) \cup \{v,p_v\}~\text{and}~E_{t+1}^+ =E(C_7^-(x_1^\star x_{2}^\star)) \cup \{z_{j_0}v, vx_{1}^\star,vp_v\}, &~\text{if}~  v \in \mathcal{Q}.
		\end{cases}
		\]
		Then $\rho_{t+1}=\frac{8}{6}$ or  $\frac{9}{7}$, and $\rho_{t+1}>\frac{5}{4}$. Now by Lemma \ref{lemma3.7}, we have $e(V_4^{t+1}) \ge \frac{5}{4}|V_4^{t+1}| - \frac{7}{4}$. 
		
		The remaining vertices $x_3^\star,\dots,x_k^\star$ in $L^3_F(w_2)$ can be added  by  the same method as before.
		Moreover, in a similar way, we can add all vertices of $L^3_F(w_j)$ for all $j\ge 3$ and obtain the set $V_5$ such that $e(V_5) \ge \frac{5}{4}|V_5| - \frac{7}{4}$. 
		Indeed, the vertex set $V_5$ satisfies Properties 1, 2, and 3.
	
		Finally, we perform Step 7. 
		Let $X_4=V_5\setminus L_P^4.$ Then $|V^{\star}|=|V_5|+|X_4|+|X_4\cap \mathcal{Q}|.$
		For any $v\in X_4,$ then $v \in L^4_P(x)$ for one $x\in L^3$ and there is one vertex $v^\prime\ne v$ such that
		$v^\prime\sim v$ and $v^\prime \in L^4_P(x).$
		Let $X_4=X_4^\prime \cup X_4^{\prime\prime}$ where $X_4^\prime=\{v\in X_4 |$ the vertex $v^\prime\not\in V_5\}$ and $X_4^{\prime\prime}=\{v\in X_4 |$ the vertex $v^\prime\in V_5\}.$ Now we have 
			\begin{equation*}
			\begin{split}
				e(V^{\star}) \ge ~  & e(V_5) + e(L^3, X_4) +e(X_4^\prime)+e(X_4^{\prime\prime},V_5\cap L^4) + |X_4\cap \mathcal{Q}|\\
				\ge ~ & \frac{5}{4}|V_5| -\frac{7}{4}+|X_4|+\frac{1}{2}|X_4^\prime|+|X_4^{\prime\prime}|+|X_4\cap \mathcal{Q}|\\
				\ge ~ & \frac{5}{4}|V_5| +\frac{5}{4}|X_4|+\frac{5}{4}|X_4\cap \mathcal{Q}|-\frac{7}{4}+
				\frac{1}{4}|X_4|-\frac{1}{4}|X_4\cap \mathcal{Q}|+\frac{1}{2}|X_4^{\prime\prime}|\\
				= ~& \frac{5}{4}|V^{\star}| -\frac{7}{4}+\frac{1}{4}(|X_4|-|X_4\cap \mathcal{Q}|)+\frac{1}{2}|X_4^{\prime\prime}|\\
				\ge ~ & \frac{5}{4}|V^{\star}| -\frac{7}{4}.
			\end{split}
		\end{equation*}
		
	Hence, such vertex set $V^{\star}$ satisfies Properties~1,~2, and~$3^+$.
		
	\subsubsection{$\delta(G) = 2$}
		Let $D_2(G)$ be the set of all the vertices of degree $2$ in $G$. We partition $D_2(G)$ into three sets:
		\[
		D_2^i(G) = \{ v \in D_2(G) : |N_G(v) \cap D_2(G)| = i \} ~~\text{for $i = 0,1$ or $2$}.
		\]
		
		Using the discharge method, we have the following lemma.
		\begin{lem}\label{discharge}
			Suppose that $G$ is  $\mathcal{C}_{[4,6]}$-saturated. If $G$ contains no vertex $\alpha_0$ belonging to one of the following four types$:$
			\begin{enumerate}[itemsep=0pt, parsep=0pt, leftmargin=*, align=left]
				\item[Type $\uppercase\expandafter{\romannumeral 1}:$] $\alpha_0 \in D_2^2(G);$
				\item[Type $\uppercase\expandafter{\romannumeral 2}:$] $N_G(\alpha_0) = \{\alpha_1, \alpha_2, \alpha_3\}$ with $\{\alpha_1, \alpha_2, \alpha_3\} \subseteq D_2^0(G);$
				\item[Type $\uppercase\expandafter{\romannumeral 3}:$] $N_G(\alpha_0) = \{\alpha_1, \alpha_2, \alpha_3\}$ with $\alpha_1 \in D_2^1(G)$ and $\alpha_2 \in D_2^0(G) \cup D_2^1(G);$
				\item[Type $\uppercase\expandafter{\romannumeral 4}:$] $N_G(\alpha_0) = \{\alpha_1, \alpha_2, \alpha_3, \alpha_4\}$ with
				$\{\alpha_1, \alpha_2, \alpha_3\} \subseteq D_2^1(G)$ and $\alpha_4 \in D_2^0(G) \cup D_2^1(G),$	
			\end{enumerate}
			then $e(G) \ge \frac{5}{4}|V(G)|.$
		\end{lem}
		
		\begin{proof}
			Assume that $G$ contains no vertex of the above four types. For each $v \in V(G)$, define its initial charge as $ch(v) = d_G(v)$. Then $e(G) = \frac{1}{2} \sum_{v \in V(G)} d_G(v) = \frac{1}{2} \sum_{v \in V(G)} ch(v).$
			
			Now we redistribute the charges according to the following rules:
			\begin{itemize}[itemsep=0pt, parsep=0pt]
				\item[(R1)] Every vertex $v \in D_2^0(G)$ receives $\frac{1}{4}$ from each of its two neighbors.
				\item[(R2)] Every vertex $v \in D_2^1(G)$ receives $\frac{1}{2}$ from its unique neighbor whose degree exceeds $2$.
			\end{itemize}
			Let $ch'(v)$ denote the final charge of each vertex $v \in V(G)$ after the redistribution.
			Then $ch'(v)\ge \frac{5}{2}$ for any vertex $v$ of degree $2$ in $G.$
			
			For any vertex $v$ of degree $3$ in $G$, if $v$ has a neighbor of degree $2$,  by 
			forbidding Types \uppercase\expandafter{\romannumeral 2} and \uppercase\expandafter{\romannumeral 3}, then
			one of  the following two possibilities can occur:

			$\bullet$ $v$ has exactly one neighbor in $D_2^1(G)$, and its other two neighbors have degree at least $3$;
				
			$\bullet$ $v$ has exactly two neighbors in $D_2^0(G)$, and its remaining neighbor has degree at least $3$.\\
			In either case, $ch'(v) \ge \frac{5}{2}$.
			
			For any vertex $u$ of degree $4$ in $G$, if $u$ has a neighbor of degree $2$, by forbidding Type \uppercase\expandafter{\romannumeral 4}, then one of the following two possibilities can occur:	

			$\bullet$ $|N_G(u) \cap D_2^1(G)| = 3$, and $d(v) \ge 3$ for  the   vertex $v \in N_G(u) \setminus D_2^1(G)$;
					
			$\bullet$ $|N_G(u) \cap D_2^1(G)| \le 2$. \\
			In either case, $ch'(u) \ge \frac{5}{2}$.
			
			For any vertex $w$ of degree at least $5$, $ch'(w) \ge ch(w) - \frac{1}{2}d_G(w) \ge \frac{5}{2}.$

			Conclusively, $
			e(G) = \frac{1}{2} \sum_{v \in V(G)} ch(v) = \frac{1}{2} \sum_{v \in V(G)} ch'(v)
			\ge \frac{5}{4}|V(G)|.$
		\end{proof}

		Furthermore, we have:
		\begin{figure}
			\centering
			\includegraphics[scale=1]{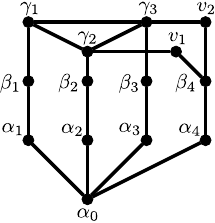}
			\caption{Graph in the proof of Lemma~\ref{lem:no-type4}}
			\label{fig_type4}
		\end{figure}	
		
		\begin{lem}\label{lem:no-type4}
			If $G$ is  $\mathcal{C}_{[4,6]}$-saturated, then $G$ has no vertex of Type \uppercase\expandafter{\romannumeral 4}.
		\end{lem}
		\begin{proof}
			By contradiction, assume that there is a  vertex $\alpha_0$ in $G$ belonging to Type \uppercase\expandafter{\romannumeral 4}, with
			$N_G(\alpha_0)=\{\alpha_1,\alpha_2,\alpha_3,\alpha_4\}$,
			$N_G(\alpha_i)=\{\alpha_0,\beta_i\}$  for $i=1,2,3,4$, and
			$N_G(\beta_j)=\{\alpha_j,\gamma_j\}$ for $j=1,2,3$.
			By Proposition \ref{prop1}, the triples $\{\alpha_0,\alpha_1,\alpha_2\}$, $\{\alpha_0,\alpha_1,\alpha_3\}$ and $\{\alpha_0,\alpha_2,\alpha_3\}$ lie on induced $C_7$, say $C_7(\alpha_1 \alpha_2) = \alpha_0 \alpha_1 \beta_1 \gamma_1 \gamma_2 \beta_2 \alpha_2 \alpha_0, C_7(\alpha_1 \alpha_3) = \alpha_0 \alpha_1 \beta_1 \gamma_1 \gamma_3 \beta_3 \alpha_3 \alpha_0$ and	$C_7(\alpha_2 \alpha_3) = \alpha_0 \alpha_2 \beta_2 \gamma_2 \gamma_3 \beta_3 \alpha_3 \alpha_0,$ respectively.
			Also, the triples $\{\alpha_0,\alpha_2,\alpha_4\}$ and $\{\alpha_0,\alpha_3,\alpha_4\}$  lie on induced $C_7$, say $C_7(\alpha_2 \alpha_4) = \alpha_0 \alpha_2 \beta_2 \gamma_2 v_1 \beta_4 \alpha_4 \alpha_0$ and $C_7(\alpha_3 \alpha_4) = \alpha_0 \alpha_3 \beta_3 \gamma_3 v_2 \beta_4 \alpha_4 \alpha_0,$ respectively. See Figure \ref{fig_type4}. 
			If $v_1 = v_2$ or $v_1 \ne  v_2$, then $v_1\gamma_2\gamma_1\gamma_3 v_1$ forms a $C_4$ or $\beta_4 v_1 \gamma_2 \gamma_3 v_2 \beta_4$  forms a $C_5$  in   $G,$ respectively, which
			 contradicts that $G$ is  $\mathcal{C}_{[4,6]}$-saturated.
		\end{proof}
		
		By Lemmas \ref{discharge} and \ref{lem:no-type4}, we can assume that the graph $G$ contains a vertex belonging to one of the Types \uppercase\expandafter{\romannumeral 1}, \uppercase\expandafter{\romannumeral 2} and \uppercase\expandafter{\romannumeral 3}.	We now construct $V^\star$ of this section through the following steps.

 	\paragraph{Step 1. Initialization: $V_0$ (see Figure \ref{figd2cases13})}   	

		$\bullet$ Type \uppercase\expandafter{\romannumeral 1}: We select one vertex $\alpha_0$ belonging to Type \uppercase\expandafter{\romannumeral 1}. Let $L^0 = \{\alpha_0\}$ and $L^1 = \{\alpha_1, \alpha_2\}$. By Proposition \ref{prop1}, there is a $C_7$ containing $\alpha_0,\alpha_1,\alpha_2,$ see the first graph in Figure \ref{figd2cases13}. Since $d_G(\alpha_i)=2$ for $i=1,2$, then $L^2= \{\alpha_3, \alpha_4\}$. Now we set $V_0=L^0 \cup L^1\cup  L^2\cup \{\beta_1,\beta_2\}.$

		$\bullet$ Type \uppercase\expandafter{\romannumeral 2}: We select one vertex $\alpha_0$ belonging to Type \uppercase\expandafter{\romannumeral 2}. Let $L^0 = \{\alpha_0\}$ and $L^1 = \{\alpha_1, \alpha_2, \alpha_3\}$. By Proposition \ref{prop1}, there are some $C_7$ containing $\alpha_0,\alpha_1,\alpha_2,$ or $\alpha_0,\alpha_1,\alpha_3,$ or $\alpha_0,\alpha_2,\alpha_3,$ see the second graph in Figure \ref{figd2cases13}. Since $d_G(\alpha_i)=2$ for $i=1,2,3$, then $L^2= \{\alpha_4,\alpha_5,\alpha_6\}$.
		Now we set $V_0=L^0 \cup L^1\cup  L^2 \cup \{\beta_1,\beta_2,\beta_3,\beta_4,\beta_5,\beta_6\}.$

		$\bullet$ Type \uppercase\expandafter{\romannumeral 3}: We select one vertex $\alpha_0$ belonging to Type \uppercase\expandafter{\romannumeral 3}. Let $L^0 = \{\alpha_1\}$ and $L^1 = \{\alpha_0, \alpha_4\}$. By Proposition \ref{prop1}, there are some $C_7$ containing  $\alpha_1,\alpha_0,\alpha_4,$  see the third graph in Figure \ref{figd2cases13}. Since $d_G(\alpha_0)=4$ and $d_G(\alpha_4)=2,$ then $L^2= \{\alpha_2,\alpha_3,\alpha_5\}$. Now we set $V_0=L^0 \cup L^1\cup  L^2\cup \{\beta_1,\beta_2,\beta_3,\beta_4\}.$

		For Types \uppercase\expandafter{\romannumeral 1}, \uppercase\expandafter{\romannumeral 2}, \uppercase\expandafter{\romannumeral 3},  it is not hard to check that $e(V_0) \ge  \frac{5}{4}|V_0| - \frac{7}{4}.$
			
					\begin{figure}
			\centering
			\includegraphics[scale=1]{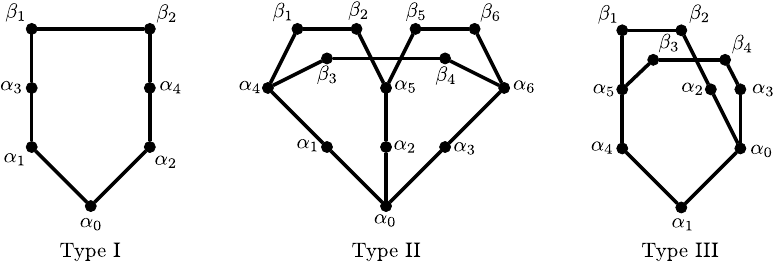}
			\caption{Initialization: $V_0$}
			\label{figd2cases13}
		\end{figure}
					
	\paragraph{Step 2.  $V_1=V_0 \cup L_P^3$}   	
		In this step, for Types \uppercase\expandafter{\romannumeral 1}, \uppercase\expandafter{\romannumeral 2}, \uppercase\expandafter{\romannumeral 3}, add all vertices of $L^3_P$ into $V_0$, obtaining $V_1$.
		Then
			\[
			|V_1| = |V_0| + |L^3_P|\quad \text{and} \quad
			e(V_1) = e(V_0) + \frac{3}{2}|L^3_P|
			\ge \frac{5}{4}|V_1| - \frac{7}{4} + \frac{1}{4}|L^3_P|.
			\]
	\paragraph{Step 3.  $V^{\star}=V_1\cup L_F^3$}   	
		This step is similar to the Step 6.1 in Section 3.1.  For each vertex $v \in L^2$, we add all vertices of $L^3_F(v)$ into $V_1,$ obtaining final set  $V^{\star}$. 
		
		Suppose that $v$ is any vertex in the $L^2$ of  Types \uppercase\expandafter{\romannumeral 1}, \uppercase\expandafter{\romannumeral 2} or \uppercase\expandafter{\romannumeral 3}, and $L_F^3(v)=\{x_1,\dots,x_t\}.$
		By Proposition \ref{prop1}, for any two non-adjacent $x_i,x_{i+1}$ in $L_F^3(v)$, three vertices 	$v,x_i,x_{i+1}$ lie in an induced $C_7(x_ix_{i+1})=vx_iy_iz_iz_{i+1}y_{i+1}x_{i+1}v$. Similarly, denote $V(C_7^-(x_i x_{i+1}))=V(C_7(x_ix_{i+1})) \setminus \{v, x_i\}$ and $E(C_7^-(x_i x_{i+1}))=E(C_7(x_ix_{i+1})) \setminus \{vx_i\}$  respectively.

		For convenience, let $V_1=V_1^0$. Similar to Step 6.1 in Section $3.1,$ we define $V_{1}^{i+1}=V_{1}^{i}\cup U_i^+$ for $i=0,1,\dots,t-1$, where $U_i^+$ contains the vertex $x_{i+1}$ if $x_{i+1}\not\in V_{1}^{i}$ and  $U_i^+=\emptyset$ if $x_{i+1}\in V_{1}^{i}$. Indeed, the final vertex set $V_{1}^{t}$ contains all vertices of $L_F^3(v)$. We still define analogously the set of newly added edges as $E_i^+= E(U_i^+) \cup E(U_i^+, V_4^i).$ With this notation, we have $e(V_4^{i+1}) \ge e(V_4^i) + |E_i^+|$. We now proceed by induction on $t$ to prove that $e(V_1^{t}) \ge \frac{5}{4}|V_1^{t}| - \frac{7}{4}$.
	
		For $t=0$, then $V_1^0=V_1$ and  result holds by Step 2. For $t=1$, note that each $v \in L^2$ has one neighbor $v^\ast\in L^3$ which has been added in Step 1. For example, if 
		$v=\alpha_3$ in Type \uppercase\expandafter{\romannumeral 1}, then $v^\ast=\beta_1$. If $v^\ast\in L_P^3(v)$, then $|L^3_P|\ge 2$  and $e(V_1)= e(V_1^0) \ge \frac{5}{4}|V_1^0| - \frac{7}{4} + \frac{1}{4}|L^3_P| \ge \frac{5}{4}|V_1^0| - \frac{5}{4}$.
		Now we set $U_0^+=\{x_1\}$ and  $E_0^+=\{v x_1\},$ and then $e(V_1^1)=	e(V_1^0)+1\ge \frac{5}{4}|V_1^0|- \frac{5}{4}+1 =\frac{5}{4}(|V_1^0|+1) - \frac{6}{4} \ge  \frac{5}{4}|V_1^1| - \frac{7}{4}.$
		If  $v^\ast\not\in L_P^3(v),$ then $v^\ast\in L_F^3(v).$ Assume that   $v^\ast=x_1,$ then  $x_1\in V_1^0.$
		Now we set $U_i^+=\emptyset$ and $E_i^+=\emptyset,$ which implies that  $V_{4}^{1}=V_{4}^{0}$ and $e(V_1^1)\ge  \frac{5}{4}|V_1^1| - \frac{7}{4}.$
		Hence, the result holds for $t=1$.
		
		By  the induction hypothesis, assume that $e(V_1^{i}) \ge \frac{5}{4}|V_1^{i}| - \frac{7}{4}$ for $2\le i\le t-1$.
		We divide into three cases to construct  suitable $U_i^+$ and $E_i^+$ such that $\rho_i \ge  \frac{5}{4}.$

		\begin{mycase}{1} 
			$V(C_7^-(x_ix_{i+1})) \cap V_1^i \ne \emptyset$. Let  $k = |V(C_7^-(x_ix_{i+1})) \cap V_1^i|\ge 1$. We set $U_i^+ = V(C_7^-(x_ix_{i+1}))\setminus V_1^i$ and $E_i^+ = E(C_7^-(x_ix_{i+1}))\setminus E(V_1^i).$ Then $\rho_i \ge \frac{6-k}{5-k} \ge \frac{5}{4}$.
		\end{mycase}
			
		\begin{mycase}{2} 
			$V(C_7^-(x_ix_{i+1})) \cap V_1^i = \emptyset$ and $\{z_i,z_{i+1}\} \nsubseteq L^5$. Without loss of generality, assume that $z_{i+1}\in L^4$ or $z_{i+1}\in L^3$. If $z_{i+1}\in L^3$, then $z_{i+1}\in L^3(u_2)$ for one vertex $u_2\in L^2$. We set
			$$U_i^+= V(C_7^-(x_ix_{i+1}))~\text{and} ~E_i^+ = E(C_7^-(x_ix_{i+1})) \cup \{z_{i+1}u_2\}.$$
			Then $\rho_i =\frac{7}{5}> \frac{5}{4}$. If $z_{i+1}\in L^4,$ then $z_{i+1}\in L^4(u_3)$ and $u_3\in L^3(u_2)$ for some vertices $u_3\in L^3$ and $u_2\in L^2$.
			We set 
			\[
			\begin{cases}
				U_i^+ =V(C_7^-(x_ix_{i+1})) ~\text{and}~E_i^+ = E(C_7^-(x_ix_{i+1})) \cup \{z_{i+1}u_3\}, &\text{if} ~ u_3\in V_1^i;\\
				U_i^+ =V(C_7^-(x_ix_{i+1}))\cup \{u_3\}~\text{and}~E_i^+ = E(C_7^-(x_ix_{i+1})) \cup \{z_{i+1}u_3,u_3u_2\}, &\text{if} ~ u_3\not\in V_1^i.
			\end{cases}
			\]
			Then $\rho_i = \frac{8}{6}$ or  $\frac{7}{5}$, which implies that $\rho_i > \frac{5}{4}$.
		\end{mycase}
			
		\begin{mycase}{3}  
			$V(C_7^-(x_ix_{i+1})) \cap V_1^i = \emptyset$ and $\{z_i,z_{i+1}\} \subseteq L^5$. Note that each $v\in L^2$ lies in a  $C_7(a_1b_1)=a_0a_1va_3b_3b_2b_1a_0$, where $a_0\in L^0,$ $a_1,b_1\in L^1,$ $b_2\in L^2,$ $a_3,b_3\in L^3,$ and   all vertices of $C_7(a_1b_1)$  have been added  in Step 1. For example, if $v=\alpha_3$ in Type \uppercase\expandafter{\romannumeral 1}, then $a_0=\alpha_0,$ $a_1=\alpha_1,$ $b_1=\alpha_2$, $b_2=\alpha_4$, $a_3=\beta_1$ and $b_3=\beta_2$.
			It is clear that $z_{i+1}\not\sim b_1.$
			 By Fact 5, there exists a path $P_{z_{i+1}b_1}$ of length $\ell \in \{3,4,5\}$ between $z_{i+1}$ and $b_1$, write $P_{z_{i+1}b_1}= z_{i+1} v_1 \dots v_{\ell-1} b_1$. Since $b_1\in L^1$, then the vertex $v_{\ell-1}\in V_1^i$. We set
		$$			U_i^+=V(C_7^-(x_ix_{i+1}))\cup \{v_1,\dots,v_{j^\ast}\}
	~\text{and}~
			E_i^+=E(C_7^-(x_ix_{i+1}))\cup \{z_{i+1}v_1,v_1v_2,\dots,v_{j^\ast}v_{j^\ast+1}\},$$
			where $j^\ast=\min \{j\in \{1,\dots,\ell-2\}| v_j\in V_1^i$ and $v_{j+1}\not\in V_1^i\}$. Then $\rho_i = \frac{6+1+ j^\ast}{5+  j^\ast} \ge \frac{5}{4}$ for $1\le j^\ast \le \ell-2.$
		\end{mycase}
				
		From the above three cases, we can always construct  suitable $U_i^+$ and $E_i^+$ such that
  		$\rho_i\ge \frac{5}{4}$. By the induction hypothesis together with Lemma~\ref{lemma3.7}, then   $e(V_1^{t}) \ge \frac{5}{4}|V_1^{t}| - \frac{7}{4}$ holds.
			
		Repeat the above procedure	for each vertex in $L^2$, we obtain the final set  $V^{\star}$, which contains all vertices of  $L_F^3$ and satisfies $e(V^{\star}) \ge \frac{5}{4}|V^{\star}| - \frac{7}{4}$. Hence, such vertex set $V^{\star}$ satisfies Properties 1, 2 and $3^+$. 

		The proof of Lemma \ref{lem_existence_of_Vstar} is finished.			
	
	\section{More on $\mathcal{C}_{[4,r]}$-saturated Graphs}
		In this paper, we prove that $\sat(n, \mathcal{C}_{[4,r]})\le \lceil\frac{5n}{4} - \frac{r+1}{4}\rceil$ for $n\ge r+1$ and $\sat(n, \mathcal{C}_{[4,6]})=\lceil\frac{5n}{4} - \frac{7}{4}\rceil$ for $n \ge 7$. It is worth noting that for $r$ sufficiently large, the bound in Theorem~\ref{thm_c4r} is not tight. For example, let $r \ge 11$ be odd and $n = \frac{rk - k + 4}{2}$ where  $k>2+\frac{48}{(r-1)(r-9)}$, we can construct an $n$-vertex $\mathcal{C}_{[4,r]}$‑saturated graph $\Theta_n$ obtained by adding $k$ edge-disjoint paths of length $\frac{r+1}{2}$ between two nonadjacent vertices $v$ and $u$. It is easy to check that $\Theta_n$ is $\mathcal{C}_{[4,r]}$‑saturated, but $e(\Theta_n) = \frac{(r+1)k}{2} < \frac{(r+1)n}{r-1} - \frac{2r-4}{r-1} < \frac{5n}{4} - \frac{r+1}{4}$.

	\appendix
	
	\section{Appendix}
	
	\begin{proof}[Proof of Lemma~\ref{lem_yi_in_L4}]
		Since $x_i \in L^3_F(x)$ and $y_i \in N_G(x_i)$, then $y_i \notin L^2$ by Proposition \ref{prop3}. If $y_i \in L^3(x^\prime)$ for one $x^\prime\in L^2$, then $x^\prime \ne x$ otherwise $x_i \in L^3_P(x)$; but now $x_i y_i x^\prime \alpha_1 x x_i$ forms a $C_5$ in $G$, a contradiction. Therefore, $y_i \in L^4$. Similarly, we have $y_{i+1} \in L^4$.
	\end{proof}

	\begin{proof}[Proof of Lemma~\ref{lem_yi_in_Q_implies_yi_in_L4p}]
	We prove only the case $y_i\in \mathcal{Q};$ the proof for $y_{i+1}$ is similar. 
		If $y_i\in \mathcal{Q},$
		since $x_i \in N_G(y_i)$, Fact 4 ensures that there is a vertex $y_i^\prime \in N_G(x_i) \cap N_G(y_i)$. By Proposition \ref{prop3} and $y_i\in L^4(x_i)$, we have $y_i^\prime \not\in L^3.$ 
		It follows that $y_i^\prime\in L^4(x_i)$ and
		 $y_i \in L^4_P(x_i)$. By Fact 2, such $y_i^\prime$ is unique.
	\end{proof}

	\begin{proof}[Proof of Lemma~\ref{lem_zi_in_L45}]
		If $z_i \in L^3$, then $\{x_i, z_i\}\subseteq N_G(y_i) \cap L^3$, which contradicts Proposition \ref{prop3}. Hence, $z_i \in L^4 \cup L^5$. Similarly, we have $z_{i+1} \in L^4 \cup L^5$.
	\end{proof}

	\begin{proof}[Proof of Lemma~\ref{lem_zi_in_Q_implies_zi_in_L4}]
			We prove only the case $z_i\in \mathcal{Q};$ the proof for $z_{i+1}$ is similar. 
 If $z_i\in \mathcal{Q}$, then $z_i\in L^4(x_i^\prime)$ for some $x_i^\prime\in L^3$ (as  $L^5\cap \mathcal{Q}=\emptyset$). Since $C_7(x_i x_{i+1})$ is an induced $C_7,$ then $z_i$ is not adjacent to  $x_i$ and $x_{i+1}$, which implies that $x_i^\prime\ne x_i, x_{i+1}.$ Hence, $x_i^\prime \in L^3\setminus \{x_i, x_{i+1}\}.$
	\end{proof}
	
	\begin{proof}[Proof of Lemma~\ref{lem_distinct}]
		We only consider that $x_i^\prime = x_{i+1}^\prime;$ the case of $x_i^\prime \ne x_{i+1}^\prime$ is similar.
		Clearly, $x_i, x_{i+1}, y_i, y_{i+1}, z_i, z_{i+1}$ are pairwise distinct.
		If any two vertices in the set  $\{x_i, x_{i+1}, y_i,$ $y_{i+1},z_i, z_{i+1}, x_i^\prime,$ $x_{i+1}^\prime, y_i^\prime, y_{i+1}^\prime\}$ coincide, then a straightforward check shows that there is a $C_4$, $C_5$, or $C_6$ in $G$, contradicting that $G$ is $\mathcal{C}_{[4,6]}$-saturated.
		 This completes the proof.
	\end{proof}
	
		\vskip 0.4 true cm
{\textbf{Declaration of competing interest}}
\vskip 0.4 true cm
The authors declare that they have no known competing financial interests or personal relationships that could have appeared to influence the work reported in this paper.

	\vskip 0.4 true cm
{\textbf{Acknowledgments}}
\vskip 0.4 true cm

This project  is supported by the National Natural Science Foundation of China (Nos. 12271484).

		\vskip 0.4 true cm
{\textbf{Data availability}}
\vskip 0.4 true cm
No data was used for the research described in the article.


\begin{thebibliography}{99}
	
		\bibitem{Che09}
		Y.C. Chen, Minimum $C_5$-saturated graphs, J. Graph Theory 61 (2) (2009) 111--126.
		
		\bibitem{Che11}
		Y.C. Chen, All minimum $C_5$-saturated graphs, J. Graph Theory 67 (1) (2011) 9--26.
		
		\bibitem{Clark83}
		L.H. Clark, R.C. Entringer, Smallest maximally nonhamiltonian graphs, Period. Math. Hung. 14 (1983) 57--68.
		
		\bibitem{Clark92}
		L.H. Clark, R.C. Entringer, H.D. Shapiro, Smallest maximally nonhamiltonian graphs II, Graphs Comb. 8 (3) (1992) 225--231.
		
			\bibitem{CF21}
		B. Currie, J. Faudree, R. Faudree, J. Schmitt, A survey of minimum saturated graphs, Electron. J. Comb. 18 (2021) \#DS19.
		
		
		\bibitem{EHM64}
		P. Erd\H{o}s, A. Hajnal, J.W. Moon, A problem in graph theory, Am. Math. Mon. 71 (1964) 1107--1110.
		
		\bibitem{Subdivision12}
		M. Ferrara, M. Jacobson, K.G. Milans, C. Tennenhouse, P.S. Wenger, Saturation numbers for families of graph subdivisions, J. Graph Theory 71 (4) (2012) 416--434.
		
		\bibitem{FFL97}
		D.C. Fisher, K. Fraughnaugh, L. Langley, $P_3$-connected graphs of minimum size, Ars Comb. 47 (1997) 299--306.
		
		\bibitem{FK13}
		Z. F\"{u}redi, Y. Kim, Cycle-saturated graphs with minimum number of edges, J. Graph Theory 73 (2) (2013) 203--215.
		
		\bibitem{Lan21}
		Y. Lan, Y. Shi, Y. Wang, J. Zhang, The saturation number of $C_6$, Discrete Math. 348 (8) (2025) 114504.
		
		\bibitem{LJZY97}
		X. Lin, W. Jiang, C. Zhang, Y. Yang, On smallest maximally non-Hamiltonian graphs, Ars Comb. 45 (1997) 263--270.
		
		\bibitem{MHHG21}
		Y. Ma, X. Hou, D. Hei, J. Gao, Minimizing the number of edges in $C_{\ge r}$-saturated graphs, Discrete Math. 344 (11) (2021) 112565.
		
		\bibitem{Ma25}
		Y. Ma, Minimum saturated graphs without 4-cycles and 5-cycles, Discrete Math. 348 (2025) 114690.
		
		\bibitem{Oll72}
		L.T. Ollmann, $K_{2,2}$ saturated graphs with minimal number of edges, in: Proceedings of Third Southeastern Conference on Combinatorics, Graph Theory, and Computing, Florida Atlantic Univ., Boca Raton, FlaBoca Raton, Fla, 1972, Florida Atlantic Univ., Boca Raton, Fla, 1972, pp. 367--392.
		
		\bibitem{Tuz89}
		Z. Tuza, $C_4$-saturated graphs of minimum size, in: 17th Winter School on Abstract Analysis, Srn\'{i}, 1989, Acta Univ. Carol., Math. Phys. 30 (2) (1989) 161--167.
		
		
	\end{thebibliography}
\end{document}